\documentclass[oneside,a4paper, english, 11pt]{amsart}

\usepackage[T1]{fontenc}
\usepackage[latin9]{inputenc}
\usepackage{amsmath}

\makeatletter
\newenvironment{abstracts}{%
  \ifx\maketitle\relax
    \ClassWarning{\@classname}{Abstract should precede
      \protect\maketitle\space in AMS document classes; reported}%
  \fi
  \global\setbox\abstractbox=\vtop \bgroup
    \normalfont\Small
    \list{}{\labelwidth\z@
      \leftmargin3pc \rightmargin\leftmargin
      \listparindent\normalparindent \itemindent\z@
      \parsep\z@ \@plus\p@
      
      \itemsep\medskipamount
    }%
}{%
  \endlist\egroup
  \ifx\@setabstract\relax \@setabstracta \fi
}

\newcommand{\abstractin}[1]{%
  \otherlanguage{#1}%
  \item[\hskip\labelsep\scshape\abstractname.]%
}
\makeatother
 \theoremstyle{plain}
\newtheorem{thm}{Theorem}[section]
\theoremstyle{plain}
  \newtheorem{prop}[thm]{Proposition}
\theoremstyle{plain}
 \newtheorem{lemma}[thm]{Lemma}
\theoremstyle{plain}

\theoremstyle{plain}
\newtheorem{cor}[thm]{Corollary}
\theoremstyle{definition}
  \newtheorem{defn}[thm]{Definition}
 \theoremstyle{definition}
  
  \newtheorem{exam}[thm]{Example}
\theoremstyle{remark}
\newtheorem{rmk}[thm]{Remark}
\numberwithin{equation}{section}
 
\usepackage{amsmath}
\usepackage{amssymb}
\usepackage{amscd}
\usepackage{amsthm}
\usepackage{array}
\usepackage{hyperref}
\usepackage[arrow,matrix,tips]{xy}

\usepackage{stmaryrd}
\usepackage{url}
\usepackage{color}

\pdfpageheight\paperheight
\pdfpagewidth\paperwidth
\newcommand{\Z}{\mathbb{Z}}

\newcommand{\Q}{\mathbb{Q}}

\newcommand{\Qp}{\mathbb{Q}_p}

\newcommand{\F}{\mathbb{F}}

\newcommand{\fM}{\mathfrak{M}}

\newcommand{\cD}{\mathcal{D}}
\newcommand{\cE}{\mathcal{E}}
\newcommand{\cF}{\mathcal{F}}
\newcommand{\cG}{\mathcal{G}}

\newcommand{\cM}{\mathcal{M}}

\newcommand{\cO}{\mathcal{O}}
\newcommand{\cP}{\mathcal{P}}

\newcommand{\eps}{\varepsilon}

\newcommand{\phz}{\varphi}

\newcommand{\La}{\Lambda}

\newcommand{\Zp}{\mathbb{Z}_p}

\newcommand{\Gal}{\mathrm{Gal}}
\newcommand{\Hom}{\mathrm{Hom}}

\newcommand{\Res}{\mathrm{Res}}
\newcommand{\Ind}{\mathrm{Ind}}

\newcommand{\GL}{\mathrm{GL}}

\newcommand{\Spec}{\mathrm{Spec}\ }

\DeclareMathOperator{\Fil}{Fil}
\DeclareMathOperator{\Id}{id}

\DeclareMathOperator{\gr}{gr}

\DeclareMathOperator{\Gr}{Gr}
\DeclareMathOperator{\Fl}{Fl}
\DeclareMathOperator{\Nilp}{Nilp}
\DeclareMathOperator{\dR}{dR}
\DeclareMathOperator{\Adm}{Adm}
\DeclareMathOperator{\cris}{cris}

\newcommand{\ra}{\rightarrow}

\newcommand{\iarrow}{\hookrightarrow}

\newcommand{\rhobar}{\overline{\rho}}

\usepackage[french, english]{babel}
\makeatother
\title{Kisin modules with descent data and parahoric local models}

\author{Ana Caraiani}
\address{Mathematisches Institut der Universit\"at Bonn, Endenicher Allee 60 \\
	D-53115 Bonn, Germany}
\email{caraiani@math.uni-bonn.de}

\author{Brandon Levin}
\address{Department of Mathematics, Eckhart Hall, University of Chicago, Chicago, Illinois 60637, USA}
\email{bwlevin@math.uchicago.edu}

\begin{document}



\begin{abstracts}
\abstractin{english}
We construct a moduli space $Y^{\mu, \tau}$ of Kisin modules with tame descent datum $\tau$ and with $p$-adic Hodge type $\leq\mu$, for some finite extension $K/\Qp$.  We show that this space is smoothly equivalent to the local model for $\Res_{K/\Qp} \GL_n$, cocharacter $\{ \mu \}$, and parahoric level structure.  We use this to construct the analogue of Kottwitz-Rapoport strata on the special fiber $Y^{\mu, \tau}$ indexed by the $\mu$-admissible set. We also relate $Y^{\mu, \tau}$ to potentially crystalline Galois deformation rings.

\abstractin{french}
Nous construisons un espace de modules $Y^{\mu, \tau}$ de modules de Kisin avec donn\'e de descente mod\'er\'ee $\tau$ et type de Hodge $p$-adique $\mu$, pour une extension finie $K/\Qp$. Nous d\'emontrons une \'equivalence lisse entre $Y^{\mu, \tau}$ et le mod\`ele local pour la restriction de scalaires $\Res_{K/\Qp} \GL_n$, co-charact\`ere $\{ \mu \}$ et structure de niveau parahorique. Cette \'equivalence est ensuite utilis\'ee pour construire l'analogue de la stratification de Kottwitz-Rapoport sur la fibre sp\'eciale de $Y^{\mu, \tau}$, param\'etr\'ee par l'ensemble des \'el\'ements $\mu$-admissibles. Nous d\'ecrivons aussi la relation entre $Y^{\mu,\tau}$ et l'espace de d\'eformations galoisiennes potentiellement cristallines.
\end{abstracts}

\selectlanguage{english}
\maketitle
\tableofcontents

\section{Introduction}

Let $K/\Qp$ be a finite extension.  Kisin~\cite{Fcrystals} showed that the category of finite flat commutative group schemes over $\cO_K$ killed by a power of $p$ is equivalent to the category of \emph{Breuil-Kisin modules} of height $\leq 1$.  While the former do not naturally live in families, one can work with Breuil-Kisin modules with coefficients and study their moduli.  The landmark paper~\cite{MFFGS} uses moduli of Breuil-Kisin modules to construct resolutions of flat deformation rings with stunning consequences for modularity lifting theorems and applications to the Fontaine-Mazur conjecture. The main result of~\cite{MFFGS} is a modularity lifting theorem in the potentially Barsotti-Tate case. One of the key points is a rather surprising connection to the theory of \emph{local models of Shimura varieties}. Kisin showed that the singularities of the moduli space of Breuil-Kisin modules of rank $n$ (with fixed $p$-adic Hodge type) could be related to the singularities of local models for the group $\Res_{K/\Qp} \GL_n$ (with maximal parahoric level) which had been studied by \cite{PR2}.   

Kisin's result is globalized in~\cite{PRcoeff}, where Pappas and Rapoport construct a global (formal) moduli stack $X^{\mu}$ of Kisin modules with $p$-adic Hodge type $\mu \in (\Z^n)^{\Hom(K, \overline{\Q}_p)}$.  They link the space $X^{\mu}$ via smooth maps with a (generalized) local model $M(\mu)$.  When $\mu$ is non-minuscule, $M(\mu)$ is not related to any Shimura variety but is nevertheless known to have nice geometric properties by work of Pappas-Zhu~\cite{PZ} and of the second author~\cite{LevinLM}.  $M(\mu)$ is constructed inside a mixed characteristic version of the Beilinson-Drinfeld affine Grassmannian. As a result, the nice geometric properties of $M(\mu)$ transfer to the global moduli stack $X^{\mu}$.     

While the connection between moduli of Breuil-Kisin modules and local models suffices for proving modularity lifting theorems in the potentially Barsotti-Tate case, it doesn't seem capture some of the more subtle aspects of the geometry of local deformation rings. These more subtle aspects are connected to the (geometric) Breuil-Mezard conjecture~\cite{BMconj, EGgeomBM}, to the weight part in Serre's conjecture~\cite{BDJ, GHS} and to questions about integral structures in completed cohomology~\cite{BreuilLG, EGS}. Therefore, there is considerable interest in generalizing the results of Kisin and Pappas-Rapoport. This paper extends the relationship with local models to the case of Breuil-Kisin modules equipped with \emph{tame descent data}. 


We explain the connection to integral structures in completed cohomology. One of the few situations where we have explicit presentations of local deformation rings is the case of tamely Barsotti-Tate deformations rings for $\GL_2$. Set $G_K:=\mathrm{Gal}(\bar K/ K)$ and let $I_K\subset G_K$ be the inertia subgroup. When $K/\Qp$ is unramified and $\tau:I_K \ra \GL_2(\La)$ is a (generic) tame \emph{inertial type}, then~\cite{BreuilLG, BM2, EGS} explicitly describe the potentially Barsotti-Tate deformation ring $R_{\rhobar}^{\mathrm{BT}, \tau}$ for any $\rhobar:G_K \ra \GL_2(\F)$.  These computations provided evidence for the Breuil-M\'ezard conjecture and led Breuil to several important conjectures \cite{BreuilLG}. Perhaps the most striking is the precise conjecture about which lattices inside the smooth $GL_2(\cO_K)$-representation $\sigma(\tau)$ (determined by $\tau$ via inertial local Langlands) can occur globally, in completed cohomology. Breuil's conjectures were proved by Emerton-Gee-Savitt~\cite{EGS} using the explicit presentations of tamely Barsotti-Tate deformation rings.  

In more general situations ($K/\Qp$ ramified or $\rhobar$ non-generic), one cannot hope for such an explicit presentation.  In this paper, we construct for arbitrary $K/\Qp$ and $\GL_n$, resolutions of tamely Barsotti-Tate deformation rings whose geometry is related to that of local models for $\Res_{K/\Qp} \GL_n$ with parahoric level structure. These resolutions are related to the moduli of Breuil-Kisin modules with descent data. The level structure is determined by the tame inertial type $\tau$. For example, if $\tau$ consists of distinct characters, then the local model will have Iwahori level structure, whereas the local models of~\cite{MFFGS, PRcoeff}, which have trivial descent data, always have maximal parahoric level.  

Our perspective in this paper is largely global, in the spirit of~\cite{PRcoeff}. Motivated by the moduli stack of finite flat representations of $G_K$ constructed by \cite{EGpict1}, we study moduli stacks $Y^{\mu, \tau}$ of Kisin modules with tame descent data and $p$-adic Hodge type $\mu \in (\Z^n)^{\Hom(K, \overline{\Q}_p)}$. We can consider a moduli stack of Kisin modules as above, but in addition equipped with an eigenbasis compatible with the descent datum; we call this space $\widetilde{Y}^{\mu,\tau}$. 

\begin{thm}\label{main thm} There exists a moduli stack $Y^{\mu,\tau}$ of Kisin modules with tame descent data and $p$-adic Hodge type $\mu$, which fits into the diagram
\begin{equation*} 
\xymatrix{
& \widetilde{Y}^{\mu, \tau} \ar[dl]_{\pi^{\mu}} \ar[dr]^{\Psi^{\mu}} & \\
Y^{\mu, \tau} & & M(\mu) \\
}
\end{equation*}
where $M(\mu)$ is the Pappas-Zhu local model \cite{PZ, LevinLM} for $(\Res_{K/\Qp} \GL_n, \mu)$ at parahoric level $($determined by $\tau)$ and both $\pi^{\mu}$ and $\Psi^{\mu}$ are smooth maps.
\end{thm}

\begin{rmk}\label{remark on commutative diagram} The key step in the construction of the local model diagram is encoded in diagram~\ref{bigdiagram}. We decompose a Kisin module $(\mathfrak{M},\phi)$ according to the descent datum and then study the interactions between the images of $\phi$ on different isotopic pieces. This is reminiscent of the classical definition of local models which involves lattice chains.   
\end{rmk}

\begin{rmk}\label{remark on the case of vector bundles} The main idea behind constructing the local model diagram in Theorem~\ref{main thm} comes by observing that there is a correspondence between having descent datum from the ramified extension $L$ down to $K$ and having a parahoric level structure defined over $K$. This relationship also appears in the theory of vector bundles over a curve, where a vector bundle with descent datum over a ramified cover of a curve corresponds to a parahoric vector bundle over the curve. 

For example, the paper~\cite{mehtaseshadri} studies the case of vector bundles over smooth projective curves $X$ over $\mathbb{C}$. Assume $X$ has genus $\geq 2$. There exists a simply connected covering of $X$ ramified at a finite set of points (the points and their ramification indices can be prescribed in advance) and this covering can be identified with the upper half space $\mathbb{H}$. We can identify $X=\mathbb{H}/\pi$, where $\pi$ is a group of automorphisms of $\mathbb{H}$ which \emph{does not} act freely on $\mathbb{H}$. Giving a vector bundle of rank $n$ on $\mathbb{H}$ with descent datum to $X$ amounts to giving the trivial rank $n$ bundle on $\mathbb{H}$ together with a homomorphism $\pi\to GL_n(\mathbb{C})$ which induces an action of $\pi$ on the trivial bundle. The invariant direct image under the projection to $X$ gives a vector bundle on $X$ together with a so-called \emph{parabolic structure}. The parabolic structure consists of assigning a flag and a set of weights to the fibre at every ramification point. This construction gives an equivalence between the category of vector bundles on $\mathbb{H}$ with descent datum to $X$ and the category of vector bundles on $X$ with parabolic structure and with rational weights. 

We also note that the more recent paper~\cite{balajiseshadri} extends the results of~\cite{mehtaseshadri} to the case where the structure group is a semisimple simply-connected algebraic group over $\mathbb{C}$ (rather than $GL_n$). 
\end{rmk}

In joint work in preparation with Emerton, Gee and Savitt~\cite{CEGS}, the first author constructs a moduli stack of two-dimensional, tamely potentially Barsotti-Tate $G_K$-representations and relates its geometry to the weight part of Serre's conjecture. In this case, the stack $Y^{\mu,\tau}$ will be a relatively explicit, partial resolution of the moduli stack of $G_K$-representations. The nice geometric properties that $Y^{\mu,\tau}$ inherits from the local model diagram turn out to be key for understanding the geometry of the latter moduli stack.  From this perspective, the present paper and the paper in preparation~\cite{CEGS} clarify the geometry which underlies a possible generalization of Breuil's lattice conjecture in the ramified setting. 


In another direction, the local model diagram above allows us to define the analogue of Kottwitz-Rapoport strata inside the special fiber of $Y^{\mu, \tau}$.  For example, if $K = \Qp$, we get locally closed substacks $\overline{Y}_w^{\mu, \tau}$ of the moduli space of mod $p$ Kisin modules with descent datum $\overline{Y}^{\mu, \tau}$ indexed by certain elements $w$ in the Iwahori-Weyl group of $\GL_n$, the so-called $\mu$-admissible elements defined by Kottwitz and Rapoport (cf. \cite[(9.17)]{PZ}).

\begin{defn}
A Kisin module $\overline{\fM} \in \overline{Y}_w^{\mu, \tau}(\overline{\F}_p)$ is said to have \emph{shape} $($or \emph{genre}$)$ $w$. 
\end{defn}
\noindent This generalizes the notion of \emph{genre} which is crucial in \cite{BreuilLG} and more recently \cite{CDM1} in describing tamely Barsotti-Tate deformation rings for $\GL_2$. 

While Kisin's resolution was most interesting when $K/\Qp$ was ramified, potentially Barsotti-Tate deformation rings have interesting geometry even when $K = \Qp$. In addition, when $n > 2$, there is an advantage to replacing weight by level and considering potentially crystalline deformation rings in questions related to Serre weight conjectures. This direction is considered in joint work in progress of the second author with B. Le Hung, D. Le and S. Morra which computes tamely crystalline deformations rings with Hodge-Tate weights $(2,1,0)$ for $K/\Qp$ unramified with applications to Serre weight conjectures for $\GL_3$ \cite{LLLM1}. The results of \cite{LLLM1} suggest close connections between the strata defined by shapes and Serre weights.      

\subsection{Overview of the paper} In Section 2, we recall the definition of local models in the sense of Pappas-Zhu, as well as the results of~\cite{PZ, LevinLM} on the geometry of local models. In Section 3, we define Kisin modules with decent data, construct the moduli space of Kisin modules with tame descent data (without imposing any conditions related to $p$-adic Hodge type) and derive the key diagram~\ref{bigdiagram}. In Section 4, we construct the local model diagram (again without imposing a $p$-adic Hodge type $\mu$) and prove that both arrows are (formally) smooth. In Section 5, we construct the stack $Y^{\mu,\tau}$, give a moduli-theoretic description of its generic fiber, describe the Kottwitz-Rapoport stratification of its special fiber and relate it to tamely potentially Barsotti-Tate Galois deformation rings.

\subsection{Acknowledgements} The idea of constructing a moduli stack of Breuil-Kisin modules with tame descent data originated in joint work of the first author with M. Emerton, T. Gee and D. Savitt, where this is done for Breuil-Kisin modules corresponding to two-dimensional, tamely Barsotti-Tate Galois representations. The idea that one should be able to relate this moduli stack to local models of Shimura varieties was suggested to us by M. Emerton, whom we thank for many useful conversations. The second author would like to thank B. Bhatt, B. Le Hung, D. Le, S. Morra for many helpful conversations. We also thank the anonymous referees for their comments and suggestions, which improved the paper. A. C. was partially supported by the NSF Postdoctoral Fellowship DMS-1204465 and NSF Grant DMS-1501064.

\subsection{Notation} 

Fix a finite extension $K/\Qp$  with $K_0$ the maximal unramified subextension. Let $f := [K_0:\mathbb{Q}_p]$ and $e_K := [K:K_0]$. Let $k$ denote the residue field of $K$, of cardinality $p^f$. Fix a uniformizer $\pi_K$ of $K$. Let $L/K$ be the totally tame extension of degree $p^f - 1$ obtained by adjoining a $(p^f -1)$st root of $\pi_K$ which we denote by $\pi_L$.  Let $W := W(k)$ be the ring of integers of $K_0$.   
 
Let $E(u) \in \Zp[u]$ be the minimal polynomial for $\pi_K$ over $\Qp$ of degree $e: = f\cdot e_K= [K:\Qp]$.  Note that $P(v) := E(v^{p^f-1}) \in \Zp[v]$ is the minimal polynomial for $\pi_L$ over $\Qp$. 
   
 Set $\Delta := \Gal(L/K)$, which is cyclic of order $p^f-1$.  We take $F$ to be our coefficient field, a finite extension of $\Qp$, with ring of integers $\La$ and residue field $\mathbb{F}$. Let $\Delta^*:=\Hom(\Delta, \La^{\times})$ be the character group.  Assume that $K_0$ embeds into $F$ and fix such an embedding $\sigma_0:K_0 \hookrightarrow F$ which induces an embedding $W \hookrightarrow \La$ and an embedding $k_0 \hookrightarrow \mathbb{F}$.   We will abuse notation and denote these all by $\sigma_0$.
 
Let $\tau:\Delta \ra \GL_n(\La)$ be a tame principal series type, i.e., $\tau\cong \oplus_{i=1}^n \chi_i$ with $\chi_i \in \Delta^*$.   We will take $ \omega_f:G_K \ra W^{\times}$ to be the fundamental character of niveau $f$ given by $\omega_f(\sigma) = \frac{\sigma(\pi_L)}{\pi_L}$.         

\section{Local models}

In this section, we recall the definition and properties of local models for the group $\Res_{K/\Qp} \GL_n$, at parahoric level and for general cocharacters.  These local models are studied in more detail and for more general groups in \cite{LevinLM}.  We will review the relevant definitions and the results we will need.  One can think of this construction as a mixed characteristic version of the deformation of the affine flag variety used by Gaitsgory in \cite{Gaitsgory}.   The strategy in mixed characteristic builds on the work of Pappas and Zhu \cite{PZ}.  For $\GL_n$, the construction originates in work of Haines and Ngo \cite{HN}.

Since $K_0$ embeds into $F$, the local models for $\Res_{K/\Qp} \GL_n$ decompose as products over the different embeddings of $K_0$ into $\overline{\Q}_p$. For now, it is convenient to fix an embedding $\sigma:K_0  \iarrow F$ and let $Q(u) := \sigma(E(u))$, an Eisenstein polynomial over $\La$. Later on, we will allow $\sigma=\sigma_0\circ \varphi^{-j}:K_0\iarrow F$, where $\sigma_0$ is the embedding we have fixed above, $\varphi$ is the lift of Frobenius on $K_0$ and $j\in \Z/f\Z$.

Fix a parabolic subgroup $P$ of $\GL_n$ over $\Spec \La$.  $P$ is the stabilizer of a filtration 
$$
0 = V_0 \subseteq V_1 \subseteq \ldots \subseteq V_{n-1} \subseteq V_n = \La^n
$$
on the free rank $n$ $\La$-module.  For any $\La$-algebra $R$ and any rank $n$ projective $R$-module $M$, a \emph{$P$-filtration} is a filtration $\{ \cF^i(M) \} \}$ which is (Zariski) locally isomorphic to $\{ V_i \otimes_{\La} R \}$.    

\begin{defn} \label{BDgrass} For any $\La$-algebra $R$, define
$$
\Gr^{Q(u)}(R) := \{\text{isomorphism classes of pairs } (L, \beta) \},
$$ 
where $L$ is a finitely generated projective $R[u]$-module of rank $n$, $\beta:L[1/Q(u)] \cong (R[u]^n)[1/Q(u)]$.  

For any $\La$-algebra $R$, define
$$
\Fl^{Q(u)}_P(R) := \{\text{isomorphism classes of triples } (L, \beta, \eps) \},
$$ 
where $(L, \beta) \in \Gr^{Q(u)}(R)$ and $\eps$ is a $P$-filtration on $L/uL$. There is a natural forgetful morphism $\mathrm{pr}:\Fl^{Q(u)}_P \ra \Gr^{Q(u)}$.

\end{defn}

We will also need some variations of these objects. There is a local version of $\Gr^{Q(u)}$:
\begin{defn} \label{BDloc} Let $\widehat{R[u]}_{(Q(u))}$ denote the $Q(u)$-adic completion of $R[u]$.  For any $\La$-algebra $R$, define
$$
\Gr^{Q(u)}_{\mathrm{loc}}(R) := \{\text{isomorphism classes of pairs } (\widehat{L}, \widehat{\beta}) \},
$$ 
where $\widehat{L}$ is a finitely generated projective $\widehat{R[u]}_{(Q(u))}$-module of rank $n$ and $\widehat{\beta}$ is a trivialization of $\widehat{L}[1/Q(u)]$.  


\end{defn}

\begin{thm}\label{comparison with local versions}  The natural map 
$$
\Gr^{Q(u)} \ra \Gr^{Q(u)}_{\mathrm{loc}} 
$$
given by $Q(u)$-adic completion is an isomorphisms of functors.  
\end{thm}
\begin{proof}  The equivalence follows from the Beauville-Laszlo descent lemma (main theorem of \cite{BLdescent}) since $Q(u)$ is a regular element of $R[u]$ for an $\La$-algeba $R$. A more general version of the descent lemma appears as Lemma 6.1 of \cite{PZ} along with more details.  
\end{proof} 

\begin{rmk} There is a local version of $\Fl^{Q(u)}_P$ as well but it requires more machinery to define. One has to work with ``parahoric'' group schemes over $\cO[u]$ as in \S 4 of \cite{PZ} or \S 3 of \cite{LevinLM}. For example, \S 6.2.1 of \cite{PZ} defines a local affine Grassmanian associated to any smooth affine group scheme over $\cO[u]$ which includes as a special case a local version of $\Fl^{Q(u)}_P$.  
\end{rmk} 

We have in fact another description of $\Gr^{Q(u)}$ and $\Fl^{Q(u)}_P$ when $p$ is nilpotent in $R$:   
\begin{defn} \label{alt} For any $\La/p^r \La$-algebra $R$, define
$$
\Gr^{Q(u)}_{\mathrm{alt}}(R) := \{\text{isomorphism classes of pairs } (L', \beta') \},
$$ 
where $L'$ is a finitely generated projective $R[\![u]\!]$-module of rank $n$ and $\beta'$ is a trivialization of $L'[1/Q(u)]$.  We define  $\Fl^{Q(u)}_{P,\mathrm{alt}}$ to include the additional data of a $P$-filtration $\eps'$ on $L'/uL'$. 
\end{defn}

\begin{prop}  Let $R$ be a $\La/p^r \La$ algebra, there are natural bijections 
$$
 \Gr^{Q(u)}(R) \xrightarrow{\sim} \Gr^{Q(u)}_{\mathrm{alt}}(R) \text{ and }  \Fl^{Q(u)}_P(R)  \xrightarrow{\sim} \Fl^{Q(u)}_{P, \mathrm{alt}}(R). 
$$
\end{prop} 
\begin{proof} When $p$ is nilpotent, the $u$-adic and $Q(u)$-adic completions of $R[u]$ are the same since $Q(u) = u^e + pQ'(u)$ and so $\Gr^{Q(u)}_{\mathrm{alt}}(R) = \Gr^{Q(u)}_{\mathrm{loc}}(R)$.  Thus, the first bijection follows from Theorem \ref{comparison with local versions}.  If $(L, \beta, \eps) \in \Fl^{Q(u)}_P(R)$ maps to $(L', \beta', \eps') \in \Fl^{Q(u)}_{P, \mathrm{alt}}(R)$, then $L/uL$ is canonically isomorphic to $L'/uL'$ and so the data of $\eps$ is equivalent to the date of $\eps'$.  
\end{proof} 

\begin{thm} The functors $\Gr^{Q(u)}$ and $\Fl^{Q(u)}_P$ are represented by ind-schemes which are ind-projective over $\Spec \La$.  
\end{thm} 
\begin{proof} This follows from Proposition 4.1.4 of~\cite{LevinLM}.
\end{proof}

Let $L_{0, R} := R[u]^n \subset (R[u]^n)[1/Q(u)]$. In this situation, we can make the ind-structure very concrete.  
\begin{defn} \label{BDlattice} For any integers $a, b$ with $b \geq a$, define
$$
\Gr^{Q(u), [a,b]}(R) = \{ (L, \beta) \in \Gr^{Q(u)}(R) \mid Q(u)^{-a} L_{0, R} \supset \beta(L) \supset Q(u)^{-b} L_{0,R}.
$$ 
Similarly, we define $\Fl^{Q(u), [a,b]}_P = \Fl^{Q(u)}_P \times_{\Gr^{Q(u)}} \Gr^{Q(u), [a,b]}$.  
\end{defn} 

\begin{prop} The functors $\Gr^{Q(u), [a,b]}$ and $\Fl^{Q(u), [a,b]}_P$ are represented by projective $\La$-schemes.
\end{prop} 
\begin{proof} See \cite[Proposition 10.1.15]{LevinThesis}.
\end{proof}

\noindent In order to describe the geometry of $\Fl_{P}^{Q(u)}(R)$, we recall the definition of the affine Grassmannian and affine flag varieties.

\begin{defn}\label{affine grassmannian} Let $\kappa$ be a field. Let $\Gr_{\GL_n}$ be the affine Grassmannian of $\GL_n$ over $\kappa$.  $\Gr_{\GL_n}$ is the ind-scheme parametrizing, 
for any $\kappa$-algebra $R$, finite projective $R[\![t]\!]$-submodules $L_R$ of $R(\!(t)\!)^n$ (we will refer to such an $L_R$ as an $R[\![t]\!]^n$-\emph{lattice} in $R(\!(t)\!)^n$).  
\end{defn}

One can also define the affine Grassmannian $\Gr_G$ for a general connected reductive group $G$ over $\kappa$. This is the fpqc quotient of group functors $G(\!(t)\!)/G[\![t]\!]$, where the loop group $G(\!(t)\!)$ sends a $\kappa$-algebra $R$ to $G(R(\!(t)\!))$. The positive loop group $G[\![t]\!]$ sends a $\kappa$-algebra $R$ to $G(R[\![t]\!])$. In the case of $GL_n$, this definition is equivalent to Definition~\ref{affine grassmannian}. The fpqc quotient $\Gr_G$ is representable by an ind-projective ind-scheme over $\kappa$. (For a general group $G$, the affine Grassmannian parametrizes $G$-bundles on $\Spec R[\![t]\!]$ together with a trivialization on $\Spec R(\!(t)\!)$, where we can think of $G$-bundles in the Tannakian sense as tensor functors from $\mathrm{Rep}_{\kappa}(G)$ to vector bundles. See, for example, Proposition 5.2 of~\cite{PZ}.) In particular, one can consider $\Gr_{\Res_{(K \otimes_{\Qp} F)/F} \GL_n}$. Over $\overline{F}$, we have a product decomposition 
$$
(\Res_{(K \otimes_{\Qp} F)/F} \GL_n)_{\overline{F}} \cong \prod_{K \iarrow \overline{F}} \GL_n.
$$
The same then holds for the affine Grassmannian, namely, 
$$
(\Gr_{\Res_{(K \otimes_{\Qp} F)/F} \GL_n})_{\overline{F}} \cong \prod_{K \iarrow \overline{F}}(\Gr_{\GL_n})_{\overline{F}}
$$
and so $\Gr_{\Res_{(K \otimes_{\Qp} F)/F} \GL_n}$ is a twisted form of $\prod_{K \iarrow \overline{F}}\Gr_{\GL_n}$.

$\Gr_{GL_n}$ has a stratification by affine Schubert cells, as follows. Fix the diagonal torus $T$ and the upper triangular Borel $B$. This induces an Bruhat ordering on the set of dominant cocharacters $\{ (d_1, d_2, \ldots, d_n) \mid d_i \geq d_{i+1} \}$ of $\GL_n$.  Let $\mu = (d_1, d_2, \ldots, d_n)$ be a dominant cocharacter. The positive loop group $\GL_n(\kappa[\![t]\!])$ acts on the affine Grassmannian $\Gr_{\GL_n}$. By the Cartan decomposition for $\GL_n(\kappa(\!(t)\!))$, the orbits of this $\GL_n(\kappa[\![t]\!])$-action are indexed by conjugacy classes of cocharacters of $\GL_n$; the orbits are called the affine Schubert cells attached to the (conjugacy classes of) cocharacters. The affine Schubert variety $S(\mu)$ is defined to be the closure of the open Schubert cell $S^\circ(\mu)$ corresponding to the conjugacy class of $\mu$. It is a finite type closed subscheme of $\Gr_{\GL_n}$.  Concretely, $S(\mu)$ parametrizes lattices whose position relative to the standard lattice are less than or equal to $\mu$ for the Bruhat-order.  In particular, $S(\mu)_{\overline{\kappa}}$ is the union of the locally closed affine Schubert cells for all $\mu' \leq \mu$ (\cite[Proposition 2.8]{RicharzSV}).   

For our chosen parabolic subgroup $P\subset \GL_n$, we recall the definition of the affine flag variety over $\F$; it will be an ind-projective scheme over $\mathbb{F}$. It will depend on our chosen embedding $\sigma: K_0\hookrightarrow F$; recall that we have defined $Q(u):=\sigma(E(u))$. 

\begin{defn}\label{affine flag variety} The affine flag variety $\mathrm{Fl}_{P_{\F}}$ associated to the pair $(\GL_n, P_{\F})$ is the moduli space of pairs $(L, \cF^{\bullet}(L/t L))$ where $L$ is a lattice in $R(\!(t)\!)^n$ and $\{ \cF^{\bullet}(L/t L) \}$ is a $P$-filtration on $L/t L$ for any $\F$-algebra $R$.   
\end{defn}

We have a forgetful map $\mathrm{Fl}_{P_{\F}} \ra \Gr_{\GL_n}$ whose fibers are isomorphic to the flag variety $\GL_n/P_{\F}$.    

\begin{prop}  The functor $\Fl^{Q(u)}_P$ is represented by an ind-projective scheme over $\Spec \La$.  Furthermore,
\begin{enumerate}
\item The generic fiber $\Fl^{Q(u)}_P[1/p]$ is isomorphic to the product $\GL_n/P_{F} \times \Gr_{\Res_{(K \otimes_{K_0,\sigma} F)/F} \GL_n}$ over $\Spec F$;
\item The special fiber $\Fl^{Q(u)}_P \otimes_{\La} \F$ is isomorphic to $\mathrm{Fl}_{P_{\F}}$.
\end{enumerate}
\end{prop}

\begin{proof}  See \cite[Proposition 2.2.8]{LevinLM}. 
\end{proof}

We now want to consider a version of $\Fl_P^{Q(u)}$ where the embedding of $K_0$ into the coefficient field $F$ is allowed to vary. Recall that we fixed such an embedding $\sigma_0:K_0\hookrightarrow F$. For each $0 \leq j \leq f-1$, view $K_0$ as a subfield of $F$ via $\sigma_j = \sigma_0 \circ \phz^{-j}:K_0\hookrightarrow F$. (Recall that $\varphi$ is the lift of Frobenius to the unramified extension $K_0/\Qp$ and $j\in \Z/f\Z$). Fix a geometric cocharacter $\mu$ of $\Res_{K/\Qp} \GL_n$  which we write as $(\mu_j)$ where $\mu_j$ is geometric cocharacter of $\Res_{K/K_0} \GL_n$ for each embedding $\sigma_j$. Furthermore, for each embedding $\sigma_j$, fix a parabolic subgroup $P_j$ of $\GL_n$.  Define the following schemes over $\Spec \La$:
$$
\Fl_K^{E(u)} := \prod_{j \in \Z/f\Z} \Fl^{E_j(u)}_{P_j} 
$$
where $E_j(u) = \sigma_j(E(u))$, and
$$
\Fl_K^{[a,b], E(u)} := \prod_{j \in \Z/f\Z} \Fl^{[a,b], E_j(u)}_{P_j}.
$$   

\begin{rmk} For now, the parabolic subgroups $P_j$ are arbitrary and they are allowed to be distinct. In Section~\ref{smooth modification}, the "shape" of the descent datum on Kisin modules will impose additional conditions on the $P_j$, which will ensure that they determine conjugate parahoric subgroups of $\GL_n$.  
\end{rmk}

For the chosen cocharacter $\mu$, we define the reflex field $F_{[\mu]}$ as the smallest subfield of $\overline{F}$ containing $F$ and over which the conjugacy class of $\mu$ is defined.  Let $\La_{[\mu]}$ denote the ring of integers of $F_{[\mu]}$. Since we have chosen $F$ to contain a copy of $K_0$, this is the union of the corresponding fields for each $\mu_j$. We will now define the local model as a scheme over $\Spec \La_{[\mu]}$; a priori it could be defined over the ring of integers in a smaller field (we just need the conjugacy class of $\mu$ and the parabolic subgroups $P_j$ to be defined over this field) but we will only need to consider the base change to  $\Spec \La_{[\mu]}$.

\begin{defn} \label{defnlocmodel} Let $S(\mu) \subset (\Gr_{\Res_{(K \otimes_{\Qp} F)/F} \GL_n})_{F_{[\mu]}}$ be the closed affine Schubert variety associated to $\{\mu\}$. For each $j\in \mathbb{Z}/f\mathbb{Z}$,  let $1_{\GL_n/P_j}$ denote the closed point of $\GL_n/P_j$ corresponding to $P_j$. Then the \emph{local model} $M(\mu)$ associated to $\mu$ is defined to be the Zariski closure of $\prod_{j\in \mathbb{Z}/f\mathbb{Z}}1_{\GL_n/P_j} \times S(\mu_j)$ in $\Fl^{E(u)}_K$. It is a flat projective scheme over $\Spec \La_{[\mu]}$.
\end{defn}

The main theorem on the geometry of local models is:
\begin{thm} \label{locmodels} The local model $M(\mu)$ is normal with reduced special fiber.  All irreducible components of $M(\mu) \otimes_{\La} \overline{\F}$ are normal and Cohen-Macaulay.
\end{thm}
\begin{proof} In this level of generality, this is Theorem 1.0.1 of~\cite{LevinLM}. This builds on Theorem 1.1 of~\cite{PZ}, where the only restriction is that $K/\Qp$ must be tamely ramified. When $\mu$ is minuscule and $P = G$, the result goes back to Theorem B of~\cite{PR1}.
\end{proof}

\begin{rmk}The proof of Theorem~\ref{locmodels} uses the coherence conjecture of Pappas and Rapoport proven by ~\cite{Zhu}. 
\end{rmk}

\begin{rmk} Xuhua He has shown in~\cite{He} that the entire local model $M(\mu)$ is Cohen-Macaulay when the $\lambda_j$ (which are defined below in (\ref{lambda})) are all minuscule.  The local model is also known to be Cohen-Macaulay when $n = 2$ (via the argument sketched at the end of~\cite{Gflatness}, using the Kottwitz-Rapoport stratification below).  
\end{rmk} 

\begin{rmk} In the case when $n=2$ and $\mu_{j,\psi}=(1,0)$ for all $j\in \Z/f\Z$ and $\psi: K\iarrow \bar F$ an embedding extending $\sigma_j$ (which is the case corresponding to tamely Barsotti-Tate Galois representations), it can be shown that the local model coincides with the standard model, defined in terms of a Kottwitz determinant condition. The key point is that the standard model at hyperspecial level is flat, as shown in~\cite{PR1}; the same holds at parahoric level and therefore the standard model coincides with the local model in the sense of~\cite{PZ}, which is obtained by taking flat closure. The upshot is that in this special case, the entire local model $M(\mu)$ has a moduli interpretation. More details on the moduli interpretation and its relationship with tamely Barsotti-Tate Galois representations will appear in~\cite{CEGS}. 
\end{rmk}


Although there is no moduli interpretation for $M(\mu)$ in general, we can describe its special fiber in terms of affine Schubert varieties inside the affine flag variety. For each $j\in \Z/f\Z$, view $K_0$ as a subfield of $F$ via $\sigma_j$ and write $\mu_j=(\mu_{j,\psi})$, where $\psi$ runs over $K_0$-embeddings $K\iarrow \bar F$. Assume that each $\mu_{j,\psi}$ is a dominant cocharacter. Define
\begin{equation} \label{lambda}
\lambda_j = \sum_{\psi:K \hookrightarrow \overline{F}} \mu_{j, \psi}.
\end{equation}
\noindent We recall the definition of the $\lambda_j$-admissible set, which was introduced by Kottwitz and Rapoport; we follow the notation and constructions of Section 2 of~\cite{LevinLM}. 

Let $\mathcal{G}_0$ be the connected reductive group scheme $\mathrm{Res}_{(\cO_K\times_{\cO_{K_0},\sigma_j}\Lambda)/\Lambda}GL_n$ over $\mathrm{Spec}\ \Lambda$ whose generic fiber is $G$. Let $\mathcal{G}:=\mathcal{G}_0\otimes_{\Lambda}\Lambda[u]$ be the constant extension. If we set $G^\flat:=\mathcal{G}_{\mathbb{F}(\!(u)\!)}$, then $\mathcal{G}_{\mathbb{F}[\![u]\!]}$ is a reductive model of $\mathcal{G}_{\mathbb{F}(\!(u)\!)}$ and the parabolic $P_j$ determines a parahoric subgroup \[\mathcal{P}_j:=\{g\in \mathcal{G}(\mathbb{F}[\![u]\!])|\ g\mod{u}\in P_{j}(\mathbb{F})\}\subset G^\flat.\] Let $\widetilde{W}$ be the Iwahori-Weyl group of the split group $G^\flat_{\mathbb{\bar F}(\!(u)\!)}$, defined as $N(\bar\F(\!(u)\!))/T^\flat_1$, where $N$ is the normalizer of a maximal torus $T^\flat$ in $G^\flat$ and $T^\flat_1$ is the kernel of the Kottwitz homomorphism for $T^\flat$ (see Section 4.1 of~\cite{PRS} for more details). $\widetilde{W}$ sits in an exact sequence 
\[0\to X_*(T^\flat)\to \widetilde{W}\to W\to 0,\] 
where $W$ is the absolute Weyl group of $(G^\flat, T^\flat)$. Define 
\[\mathrm{Adm}(\lambda_j):=\{w\in \widetilde{W}|w\leq t_\lambda, \lambda\in W\cdot\lambda_j\}.\] The order $\leq$ used in the definition of $\mathrm{Adm}(\lambda_j)$ is the Bruhat order. Let $W_{P_j}\subset W$ be the subgroup corresponding to the parahoric $\mathcal{P}_j$. Define \[\mathrm{Adm}_{P_j}(\lambda_j):=W_{P_j}\mathrm{Adm}(\lambda_j)W_{P_j}.\]
Note that the $\mathrm{Adm}(\lambda_j)$ only depends on the geometric conjugacy class of $\lambda_j$.  

\begin{thm}\label{lm special} The geometric special fiber $\overline{M}(\mu)_{\mathbb{\bar F}}$ can be identified with the reduced union of a finite set of affine Schubert varieties in the affine flag variety $\Fl_{K,\mathbb{\bar F}}^{E(u)}$. 
Hence we have a stratification
$$ 
\overline{M}(\mu)_{\overline{\F}} = \bigcup_{ (\widetilde{w}_j) \in \prod_{j=0}^{f-1} \mathrm{Adm}_{P_j}(\lambda_j) } \prod_j S^{\circ}(\widetilde{w}_j)
$$
by locally closed reduced subschemes, where $S^\circ(\widetilde{w}_j)$ is an open affine Schubert cell and these are indexed by $j$ and by the admissible set $\mathrm{Adm}_{P_j}(\lambda_j)$.  
\end{thm}

\begin{rmk} The irreducible components of $\overline{M}(\mu)_{\overline{\F}}$ are indexed by the extremal elements of $\prod_{j=0}^{f-1} \mathrm{Adm}_{P_j}(\lambda_j)$ which are in bijection with the orbit of $(\lambda_j)$ under the Weyl group $\prod_j W_{P_j}$.
\end{rmk}

\begin{proof} This follows (by taking a product over the embeddings $\sigma_j$) from Theorem 8.3 of~\cite{PZ} when $K/\Qp$ is tamely ramified and Theorem 2.3.5 of~\cite{LevinLM}, when $K/\Qp$ is wildly ramified. 
\end{proof}

Finally, we recall a generalization of the loop group which acts on $M(\mu_j)$ and on $\Fl^{E_j(u)}_{P_j}$.    Define the pro-algebraic group $L^{+, E_j(u)} \GL_n$ over $\Spec \La$ by  
$$
L^{+, E_j(u)} \GL_n (R) = \varprojlim_r  \GL_n(R[u]/E_j(u)^r) = \varprojlim_{r} \Res_{(\La[u]/E_j(u)^r)/ \La} \GL_n (R). 
$$
We define a subgroup of $L^{+, E_j(u)} \GL_n$ by 
$$
L^{+, E_j(u)} \cP_j (R) := \{ g \in L^{+, E_j(u)} \GL_n (R) \mid g \mod u \in P_j(R) \}.
$$
Similarly, for any positive integer $r$, let 
$$
\cP_{j, r} := \{ g \in \Res_{(\La[u]/E_j(u)^r)/\La} \GL_n(R) \mid g \mod u \in P_j(R) \}.
$$

\begin{prop} \label{parahoricgroups} For any positive integer $r$, the functor $\cP_{j,r}$ is represented by a smooth, geometrically connected, group scheme of finite type over $\La$.  The functor $L^{+, E_j(u)} \cP_j$ is represented by an affine group scheme $($not of finite type$)$ over $\La$ which is formally smooth over $\La$.  
\end{prop}
\begin{proof} This is consequence of some general properties about Weil restriction along finite flat morphisms. The fact that $\cP_{j, r}$ is smooth is a consequence of Proposition A.5.2(4) in \cite{PRed}.  The group scheme $\cP_{j, r}$ has geometrically connected fibers by Proposition A.5.9 in \cite{PRed}.   
\end{proof}  

\begin{prop} \label{niceaction} The group $\prod_{j \in \Z/f\Z}L^{+, E_j(u)} \cP_j$ acts on $\Fl_K^{E(u)}$.  For any cocharacter $\mu$, $M(\mu)$ is stable and the action of  $\prod_{j \in \Z/f\Z}L^{+, E_j(u)} \cP_j$ on $M(\mu)$ factors through $\prod_{j \in \Z/f\Z} \cP_{j, N}$ for some $N$ sufficiently large.  
\end{prop} 
\begin{proof} Choose $a, b$ such that $M(\mu) \subset \prod_{j \in \Z/f\Z} \Fl^{E_j(u), [a,b]}_{P_j}$.  The action of  $\prod_{j \in \Z/f\Z}L^{+, E_j(u)} \cP_j$ on  $\prod_{j \in \Z/f\Z} \Fl^{E_j(u), [a,b]}_{P_j}$ is through the group scheme  $\prod_{j \in \Z/f\Z} \cP_{j, r}$ for $r = b-a$.  Since $\prod_{j \in \Z/f\Z} \cP_{j, r}$ is flat (even smooth) over $\La$ by Proposition \ref{parahoricgroups}, stability of $M(\mu)$ follows from the fact that the generic fiber $S(\mu)$ is a union of orbits for the loop group of $\Res_{K/\Qp} \GL_n$.  
\end{proof} 
 
\begin{rmk} An action of pro-algebraic group on a ind-scheme which satisfies the property in Proposition \ref{niceaction} is ''nice'' in the sense of \cite{Gaitsgory}.
\end{rmk}

\section{Kisin modules with descent datum}\label{Kisin modules}

In this section, we will consider moduli of Kisin modules of finite height for the field $K$ together with tame descent datum for $L/K$. We work over the category $\Nilp_{\La}$ of $\La$-algebras $R$ on which $p^N = 0$ for some $N \gg 0$.  

If $R$ is such an algebra, then $(W\otimes _{\Zp}R)[\![v]\!]$ has an $R$-linear action of $\varphi$, defined by (the lift of) Frobenius on $W$ and $\varphi(v)=v^p$.  

Recall that $\Delta = \Gal(L/K)$ is a cyclic group of order $p^f-1$.  For any $g \in \Delta$ and any $R \in \Nilp_{\La}$, we let $\widehat{g}$ be the automorphism of $(W \otimes_{\mathbb{Z}_p} R)[\![v]\!]$ given by $v \mapsto (g(\pi_L)/\pi_L \otimes 1)v = (\omega_f(g) \otimes 1) v$, which acts trivially on the coefficients.   

We have a decomposition $W\otimes_{\mathbb{Z}_p}\Lambda \simeq \oplus_{j=0}^{f-1} \Lambda$, where $\sigma_j = \sigma_0 \circ \phz^{-j}:W \iarrow \La$ corresponds to the projection onto the $j$th factor in the direct sum decomposition. We will generally consider $j$ modulo $f$. For any $R\in \Nilp_{\La}$, we get an induced decomposition 
$$
(W \otimes_{\mathbb{Z}_p} R)[\![v]\!] \cong \oplus_{j \in  \Z/f\Z}  R[\![v]\!].
$$
Under this isomorphism, we have $\widehat{g}(v) = (\sigma_0 \circ \omega_f(g), \sigma_1 \circ \omega_f(g), \sigma_2 \circ \omega_f(g), \ldots, \sigma_{f-1} \circ \omega_f(g)) v$.  

Similarly, for any $(W \otimes_{\mathbb{Z}_p} R)[\![v]\!]$-module $M$, we write 
$$
M =\oplus_{j \in  \Z/f\Z} M^{(j)}.
$$
for the induced decomposition of $M$. Each $M^{(j)}$ is an $R[\![v]\!]$-direct summand of $M$.  

\begin{defn} \label{action} Let $\fM_R$ be an $(W \otimes_{\mathbb{Z}_p} R)[\![v]\!]$-module.  A \emph{semilinear action} of $\Delta$ on $\fM_R$ is collection of $\widehat{g}$-semilinear bijections $\widehat{g}:\fM_R \ra \fM_R$  for each $g \in \Delta$ such that
$$
\widehat{g} \circ \widehat{h} = \widehat{gh}
$$
for all $g, h \in \Delta$. 
\end{defn}


Note that $P(v)$, the minimal polynomial for $\pi_L$, is fixed by $\widehat{g}$ for all $g$.  Thus, $((W \otimes_{\mathbb{Z}_p} R)[\![v]\!])[1/P(v)]$ inherits a semilinear action of $\Delta$ for any $R \in \Nilp_{\La}$.   

\begin{defn} \label{defnheight} Let $R$ be any $\La$-algebra.  A \emph{Kisin module} (with bounded height) over $R$ is a finitely generated projective $(W \otimes_{\Zp} R)[\![v]\!]$-module $\fM_R$, which is Zariski locally on $\mathrm{Spec}\ R$ finite free of constant rank over $(W \otimes_{\Zp} R)[\![v]\!]$, together with an isomorphism $\phi_{\fM_R}:\phz^*(\fM_R)[1/P(v)] \cong \fM_R[1/P(v)]$. 

We say that $(\fM_R, \phi_{\fM_R})$ has \emph{height in $[a, b]$} if
$$
P(v)^{a} \fM_R \supset \phi_{\fM_R}(\phz^*(\fM_R)) \supset P(v)^{b} \fM_R
$$
as submodules of $\fM_R[1/P(v)]$.  
\end{defn}   

\begin{defn}  A \emph{Kisin module with descent datum} over $R$ is a Kisin module $(\fM_R, \phi_{\fM_R})$ together with a semilinear action of $\Delta$ given by $\{ \widehat{g} \}_{g \in \Delta}$ which commutes with $\phi_{\fM_R}$, i.e., for all $g \in \Delta$,
$$
\phz^*(\widehat{g}) \circ \phi_{\fM_R} = \phi_{\fM_R} \circ \widehat{g}.
$$
\end{defn}


Fix integers $[a,b]$ with $a \leq b$ and a positive integer $n$.  We take $X^{[a,b]}$ to be the fpqc stack over $\Nilp_{\La}$ such that $X^{[a,b]}(R)$ is the category  of Kisin modules over $R$ of rank $n$ with height in $[a,b]$, with pullback defined in the obvious way (see \S 2.a in \cite{PRcoeff}). Similarly, we define the fpqc stack $Y^{[a,b],\Delta}$, where $Y^{[a,b], \Delta}(R)$ is the category of Kisin modules of rank $n$ with descent datum over $R$ and height in $[a,b]$.  We will need some auxiliary spaces as well. 

\begin{defn} Fix $N > b - a$.   Let $\widetilde{X}^{[a,b]}$ be the fpqc stack over $\Nilp_{\La}$ given by 
$$
\widetilde{X}^{[a,b]}(R) := \{ (\fM_R, \alpha_R) \mid \fM_R \in X^{[a,b]}(R), \alpha_R:\fM_R \cong  R[\![v]\!]^n \mod P(v)^N \}.
$$
There is also an infinite version:
$$
\widetilde{X}^{[a,b], (\infty)}(R) := \{ (\fM_R, \alpha_R) \mid \fM_R \in X^{[a,b]}(R), \alpha_R:\fM_R \cong  R[\![v]\!]^n \}.
$$
\end{defn}

We leave out $N$ from the notation $\widetilde{X}^{[a,b]}$, though of course the stack does depend on $N$. The natural maps $\widetilde{X}^{[a,b]} \ra X^{[a,b]}$ (resp. $\widetilde{X}^{[a,b], (\infty)} \ra X^{[a,b]}$) are formally smooth.  For any $r \geq 1$, set
$$
X^{[a,b]}_r := X^{[a,b]} \otimes_{\La} \La/p^r \text{ and } Y^{[a,b], \Delta}_r := Y^{[a,b], \Delta} \otimes_{\La} \La/p^r.
$$

\begin{thm} \label{reprofspace} For any $r \geq 1$, $X^{[a,b]}_r$ is representable by an Artin stack of finite type over $\Spec \La/p^r$.  Furthermore,  $\widetilde{X}^{[a,b]}_r$ is represented by a scheme of finite type over $\Spec \La/p^r \La$. 
\end{thm}
\begin{proof}  The first statement follows from \cite[Theorem 2.1]{PRcoeff} as does the fact that $\widetilde{X}^{[a,b]}_1$ is represented by a finite type scheme.  Since the inclusion $\widetilde{X}^{[a,b]}_1 \subset \widetilde{X}^{[a,b]}_r $ is a nilpotent thickening, $\widetilde{X}^{[a,b]}_r$ is also a represented by a scheme by Lemma 87.3.8 of~\cite{stacks} based on the corresponding fact for algebraic spaces and on the fact that a thickening of a scheme in the category of algebraic spaces is a scheme, which is Corollary 8.2 of~\cite{Rydh}.
\end{proof}

We will return to this theorem with descent datum in Theorem~\ref{PRdescent}. First, we discuss the \emph{Galois type} or tame type of a Kisin module with descent datum.   Let $\fM_R$ be a Kisin module with descent datum over $R$.   Write 
$$
\fM_R =  \bigoplus_{j \in \Z/f\Z} \fM_R^{(j)}.
$$
We get a semilinear $\Delta$-action on $\fM^{(j)}_R$, on the Frobenius pullback $\phz^*(\fM_R^{(j)})$ as well as on the reduction $\phz^*(\fM^{(j)}_R) /v \phz^*(\fM^{(j)}_R)$.

\begin{defn} \label{Galoistype} Let $\fM_R \in Y^{[a,b], \Delta}(R)$ and set $D^{(j)}_R := \fM^{(j)}_R/v \fM^{(j)}_R$. Then we say that $\fM_R$ has \emph{type} $\tau= \oplus_{i=1}^n \chi_i$, with $\chi_i\in \Delta^*$, if for all $j \in \Z/f \Z$
$$
D^{(j)}_R \cong \tau
$$
as linear representations of $\Delta$.  
\end{defn}

\begin{rmk} In Definition~\ref{Galoistype}, we require that the type be the same for all $j\in \{0,\dots,f-1\}$. If $R=\Lambda$, the fact that $\phi_R$ commutes with the descent datum implies that the type must be the same on each component $D^{(j)}_R$. 

However, this need not always hold if $R=\F$. For example, take $f=2, K_0=K=\mathbb{Q}_{p^2}$; $L$ will be a ramified extension of $K$ obtained by adjoining a $(p^2-1)$st root of $p$. Let $\fM=\fM^{(1)}\oplus \fM^{(2)}$ be a rank $1$ Kisin module over $\F$, with $e^{(i)}$ a generator for $\fM^{(i)}$ for $i=1,2$. In $\F[\![v]\!]$, we have $P(v)=v^{p^2-1}$. Then we can have $\mathrm{Gal}(L/K)$ act on $e^{(i)}$ by $\omega_2^{i}$ and we can simply take $\phi_\F(e^{(1)})=v^{p^2-2}\cdot e^{(2)}$ and $\phi_\F(e^{(2)})=v\cdot e^{(1)}$. Then $\fM$ is a rank $1$ Kisin module with height in $[0,1]$. 

Since we are ultimately interested in relating the Kisin modules with tame descent data to Galois representations over $F$ (see Section~\ref{connection to Galois}), we do not lose anything from imposing this condition. 
\end{rmk}

\begin{prop} If $\fM_R$ is a Kisin module with descent datum of type $\tau$, then 
$$
^{\phz} D^{(j)}_R := \phz^*(\fM^{(j)}_R)/v \phz^*(\fM^{(j)}_R) \cong  \tau.
$$
\end{prop}
\begin{proof} 
The natural $R$-linear injection $\fM_R^{(j)} \ra \phz^*(\fM^{(j)}_R)$ given by $m \mapsto 1 \otimes m$ is $\Delta$-equivariant and induces an isomorphism modulo $v$.
\end{proof}

\begin{prop} \label{consttype} Let $\fM_R \in Y^{[a,b], \Delta}(R)$. Consider
$$
D^{(j)}_R = \oplus_{\chi \in \Delta^*} D^{(j)}_{R, \chi}
$$
where $D^{(j)}_{R, \chi}$ is the $\chi$-isotypic piece.  Then $D^{(j)}_{R, \chi}$ is a finite projective $R$-module and hence the rank of $D^{(j)}_{R, \chi}$  is locally constant on $\Spec R$.  
\end{prop} 

\begin{defn} Let $Y^{[a,b], \tau}$ be fpqc stack of Kisin modules with height in $[a,b]$ and descent datum of type $\tau$ over $\Nilp_{\La}$.
\end{defn}  

\begin{cor} The inclusion $Y^{[a,b], \tau} \subset Y^{[a,b], \Delta}$ is a relatively representable open and closed immersion. 
\end{cor}
\begin{proof}
This follows from Proposition \ref{consttype} which says that the type of a Kisin module with descent is Zariski locally constant. 
\end{proof}

Define $Y^{[a,b], \tau}_r := Y^{[a,b], \tau} \times_{\La} \La/p^r \La$. In the next section, we will construct a smooth cover of $Y^{[a,b], \tau}_r$ and show that it is representable by an Artin stack of finite type (Theorem~\ref{PRdescent}).  We will also relate these moduli spaces of Kisin modules with descent datum to the local models from the previous section. 

First, we will need a few preliminaries. Recall that $\tau = \oplus_{i=1}^n \chi_i$.   We can write $\chi_i$ uniquely as
$$
\chi_i = (\sigma_0 \circ \omega_f)^{\mathbf{a}_i}
$$
where $\mathbf{a}_i = a_{i, 0} + a_{i,1} p + \ldots + a_{i, f-1} p^{f-1}$.    

\begin{defn} Let $\mathbf{a}_i$ be as above. For $j \in \Z/f\Z$ define
$$
\mathbf{a}_i^{(j)} = \sum_{k =0}^{f-1} a_{i, f- j + k} p^{k}
$$
where the subscript $f - j + k$ is taken modulo $f$.  
\end{defn} 

We have chosen a global ordering on the characters $\chi_1, \chi_2, \ldots, \chi_n$.  However, it will be useful to choose a possibly different ordering at each place $j \in \Z/f\Z$.  

\begin{defn} \label{orient} An \emph{orientation} of the type $\tau$ is a set of elements $(s_j \in S_n)_{j \in \Z/f\Z}$ such that
$$
\bf{a}_{s_j(1)}^{(j)} \leq  \bf{a}_{s_j(2)}^{(j)}  \leq \bf{a}_{s_j(3)}^{(j)} \leq \ldots \leq \bf{a}_{s_j(n)}^{(j)}.
$$
In other words, if we set $\lambda_{\bf{a}}^{(j)} := {(\bf{a}^{(j)}_i)} \in \Z^n$ thought of as a cocharacter of the diagonal torus of $\GL_n$, then $s_j$ is a permutation such that $s_j^{-1}(\lambda_{\bf{a}}^{(j)})$ is anti-dominant with respect to the upper triangular Borel subgroup. 
\end{defn}

\begin{rmk}
\begin{enumerate}  
\item If the characters $\chi_i$ are pairwise distinct, then there is a unique orientation for $\tau$.  
\item For a different choice of global ordering, the set of possible orientations changes by diagonal conjugation by $S_n$.  
\item One may also be interested in studying the case when $\tau$ is an inertial type over $\mathbb{Q}_p$ which does not correspond to a principal series type (in the sense of Bushnell and Kutzko) under the inertial local Langlands correspondence. (See Theorem of~\cite{ceggps} for a general statement of inertial local Langlands, originally due to Henniart in the $2$-dimensional case.) 

For such an inertial type $\tau$ over $\Qp$, one can consider the base change $\tau'$ to $\Q_{p^f}$  where the type corresponding to $\tau$ (in the sense of Bushnell and Kutzko) becomes a principal series representation. Then $\tau'$ decomposes as a direct sum of characters, so it is a type in the sense of Definition~\ref{Galoistype}. The orientations for $\tau'$ reflect what sort of type $\tau$ was (see Example~\ref{cuspidal} below). 
\end{enumerate}
\end{rmk}  

\begin{exam} \label{cuspidal} Here we give the example of $2$-dimensional principal series and cuspidal tame types over $\mathbb{Q}_p$. (These correspond to types in principal series, and respectively supercuspidal representations, of $GL_2(\mathbb{Q}_p)$ under inertial local Langlands.) Let $0 \leq a < b < p-1$.   Consider the two dimensional tame types over $\Qp$ given by $\tau_1 = \omega_1^a \oplus \omega_1^b$ and $\tau_2 = \Ind(\omega_2^{a + pb})$.   The base changes to $\Q_{p^2}$ are
$$
\tau_1' = \omega_2^{a + a p} \oplus \omega_2^{b + bp} \text{ and } \tau_2' = \omega_2^{a + pb} \oplus \omega_2^{b + ap}
$$
respectively.  The unique orientation for $\tau_1'$ is $(\Id, \Id)$, and the unique orientation for $\tau_2'$ is $(s, \Id)$ where $s$ is the non-trivial transposition in $S_2$. 
\end{exam}

Consider the map $R[\![u]\!] \ra R[\![v]\!]$ given by $u \mapsto v^{p^f-1}$.  If $\fM_R$ is a Kisin module over $R$ with descent datum, then for each $j$, $\fM_R^{(j)}$ considered as an $R[\![u]\!]$-module has a linear $\Delta$-action and so for any $\chi \in \Delta^*$, we can consider the submodules
$$
\fM_{R, \chi}^{(j)} = \{ m \in \fM_R^{(j)} \mid \widehat{g}(m) = \chi(g) m \}
$$
for all $g \in \Delta$.  Note that $\fM_{R}^{(j)}  = \oplus_{\chi \in  \Delta^*} \fM_{R, \chi}^{(j)}$ as $R[\![u]\!]$-modules, since the order of $\Delta$ is prime to $p$.   

Similarly, we can define 
$$
^{\phz} \fM_{R, \chi}^{(j)} := \{ m \in \phz^*(\fM^{(j)}_R) \mid  \widehat{g}(m) = \chi(g) m \}.
$$
Since the descent datum commutes with the Frobenius action, we get linear maps
$$
\phi^{(j-1)}_{R, \chi}:^{\phz} \fM_{R, \chi}^{(j-1)} \ra \fM_{R, \chi}^{(j)}.
$$

\begin{rmk}
The $\chi$-isotypic piece of $\phz^*(\fM^{(j)}_R)$ is not isomorphic to $\phz^*(\fM^{(j)}_{R, \chi})$. Thus, $\phi^{(j)}_{R}$ does not define a semilinear Frobenius from $\fM^{(j-1)}_{R,\chi}$ to $\fM^{(j)}_{R, \chi}$.  This is why we denote the $\chi$-isotypic component by $^{\phz} \fM^{(j)}_{R, \chi}$. 
\end{rmk}

\begin{prop} \label{isopieces} Let $\fM_R$ be a Kisin module over $R$ of rank $n$ with descent datum. For any $j \in \Z/f\Z$ and $\chi \in \Delta^*$ the modules $\fM^{(j)}_{R, \chi}$ are finite projective $R[\![u]\!]$-modules of rank $n$. Furthermore multiplication by $v$ on $\fM_R^{(j)}$ induces an injective $R[\![u]\!]$-module homomorphism
$$
\fM^{(j)}_{R, \chi} \xrightarrow{v} \fM^{(j)}_{R, (\sigma_j \circ \omega_f) \chi}.
$$
\end{prop} 

\begin{proof} The module $\fM_{R,\chi}^{j}$ is finite projective $R[\![u]\!]$-module because it is a direct summand of the finite projective module $\fM^{j}_R$; this also implies that multiplication by $v$ on $\fM_R^{j}$ is injective.  By the discussion before Definition \ref{action}, multiplication by $v$ sends the $\chi$-isotypic piece of $\fM_R^j$ to the $(\sigma_j\circ\omega_f)\chi$-isotypic piece. The rank computation is immediate. 
\end{proof}

\begin{lemma} \label{vheight}  Let $\fM_R$  be a Kisin module with descent datum.  Let $E_j(u) := \sigma_j(E(u))$. For each $\chi \in \Delta^*$, the Frobenius on $\fM_R$ induces an isomorphism $\phi^{(j-1)}_{R, \chi}:^{\phz} \fM_{R, \chi}^{(j-1)}[1/E_j(u)] \ra \fM_{R, \chi}^{(j)}[1/E_j(u)]$  such that
$$
E_j(u)^{a}  \fM_{R, \chi}^{(j)} \supset \phi^{(j-1)}_{R, \chi} (^{\phz} \fM_{R, \chi}^{(j-1)}) \supset E_j(u)^{b} \fM_{R, \chi}^{(j)}
$$
whenever $\fM_R$ has $P(v)$-height in $[a,b]$. 
\end{lemma}
\begin{proof}
For each $j$, the map $\phi^{(j-1)}_R:\phz^*(\fM^{(j-1)}_R)[1/\sigma_j(P(v))] \cong \fM^{(j)}_R[1/\sigma_j(P(v))]$ is a $\Delta$-equivariant isomorphism.    Using that $P(v) = E(v^{p^f-1}) = E(u)$, we see that multiplication by $E(u)$ respects the decomposition into isotypic pieces.  The height condition is easy to verify.
\end{proof}

Choose an orientation $(s_j)$ for $\tau$ as in Definition \ref{orient}.  We then have the following commutative diagram for each $j$:
\begin{equation} \label{bigdiagram}
\xymatrix{
^{\phz} \fM^{(j-1)}_{R, \chi_{s_j(n)}} \ar[r] \ar[d]^{\phi^{(j-1)}_{R, \chi_{s_j(n)} }} & ^{\phz} \fM^{(j-1)}_{R, \chi_{s_j(1)}} \ar[r] \ar[d]^{\phi^{(j-1)}_{R, \chi_{s_j(1)}}} & ^{\phz} \fM^{(j-1)}_{R, \chi_{s_j(2)}} \ar[r] \ar[d]^{\phi^{(j-1)}_{R, \chi_{s_i(2)}}} & \cdots \ar[r] & ^{\phz} \fM^{(j-1)}_{R, \chi_{s_j(n-1)}} \ar[r] \ar[d]^{\phi^{(j-1)}_{R, \chi_{s_j(n-1)}}} & ^{\phz} \fM^{(j-1)}_{R, \chi_{s_j(n)}} \ar[d]^{\phi^{(j-1)}_{R, \chi_{s_j(n)}}} \\
\fM^{(j)}_{R, \chi_{s_j(n)}} \ar[r] & \fM^{(j)}_{R, \chi_{s_j(1)}} \ar[r] & \fM^{(j)}_{R, \chi_{s_j(2)}} \ar[r] & \cdots \ar[r] & \fM^{(j)}_{R, \chi_{s_j(n-1)}} \ar[r] & \fM^{(j)}_{R, \chi_{s_j(n)}}. \\
}
\end{equation}
All the maps in the diagram are injective.  The composition across each row is multiplication by $u$. The first horizontal arrow in each row is induced by multiplication $v^{p^f -1 - \bf{a}^{(j)}_{s_j(n)}  +  \bf{a}^{(j)}_{s_j(1)} }$ .  The other horizontal arrows are induced by  multiplication by $v^{\bf{a}^{(j)}_{s_j(k+1)}  -  \bf{a}^{(j)}_{s_j(k)}}$ for each $1 \leq k \leq n-1$.  If some of the $\{ \chi_i \}$ are equal, some of the maps will be the identity.  

The diagram should remind one of the diagrams that appear in the classical definition of local models for $\GL_n$ with parahoric level structure, which involve lattice chains (see~\cite{PR2} as well as Section 2 of~\cite{PRS}).   Once we have chosen an appropriate trivialization of $ \fM^{(j)}_{R, \chi_{s_j(n)}}$ in the next section the above diagram will determine an $R$-point of an appropriate local model. 


\section{Smooth modification}\label{smooth modification}

We maintain the conventions from the previous section.  In particular, we fix an ordering $\{ \chi_i \}_{i=1}^n$ of the characters appearing in $\tau$. We would like to package the data of diagram (\ref{bigdiagram}) in a different way so that the relationship to the local models from \S 2 becomes clearer. If $D$ is an $R$-module, then by a filtration on $D$, we always mean by submodules which are direct summands of $D$.  We will work with increasing filtrations.  

We continue to work over the category $\Nilp_{\La}$ of $\La$-algebra on which $p$ is nilpotent. We make the following definition:
\begin{defn} Let $X, X'$ be fpqc stacks on $\Nilp_{\La}$.  A morphism $f:X \ra X'$ is \emph{smooth} if $f \mod p^N$ is smooth for all $N \geq 1$. 
\end{defn}


\begin{defn} \label{eigenbasis} Let $\fM_R \in Y^{[a,b], \tau}(R)$.   An \emph{eigenbasis} for $\fM_R$ is a collection of bases $\beta^{(j)} = \left \{f_1^{(j)}, f_2^{(j)}, \ldots, f_n^{(j)} \right \}$  for each $\fM_R^{(j)}$ such that $f_i^{(j)} \subset \fM_{R, \chi_i}^{(j)}$.  An eigenbasis modulo $P(v)^N$ is a collection of bases $\left \{ \beta^{(j)}_N \right \}_{j \in \Z/f\Z}$ for each $\fM^{(j)}_R/ \sigma_j(P(v))^N \fM^{(j)}_R$ compatible, as above, with the descent datum. 
\end{defn} 

An eigenbasis exists whenever $D_R^{(j)}$ is free over $R$ since one can lift a basis for $D_R^{(j)}$ to $\fM_R^{(j)}$.  In particular, such a basis exists Zariski locally on $\Spec R$ for any $\fM_R \in Y^{[a,b], \tau}(R)$.   
 
\begin{defn} Fix $N > b- a$. Let $\widetilde{Y}^{[a,b], \tau}$ be the fpqc stack over $\Nilp_{\La}$ given by 
$$
\widetilde{Y}^{[a,b], \tau}(R) := \left \{ \left (\fM_R, \beta^{(j)}_N \right) \mid \fM_R \in Y^{[a,b], \tau}(R), \beta^{(j)}_N:\fM^{(j)}_{R} \cong  R[\![v]\!]^n \mod \sigma_j(P(v))^N \right \}
$$
where $(\beta^{(j)}_N)$ is an eigenbasis. We also have an infinite version given by 
$$
\widetilde{Y}^{[a,b], \tau, (\infty)}(R) := \left \{ \left(\fM_R, \beta^{(j)} \right) \mid \fM_R \in Y^{[a,b], \tau}(R), \beta^{(j)}:\fM_R^{(j)} \cong  R[\![v]\!]^n \right \}
$$
where $(\beta^{(j)})$ is an eigenbasis. 
\end{defn} 

We leave out $N$ from the notation $\widetilde{Y}^{[a,b],\tau}$, though of course the stack does depend on $N$. See Proposition~\ref{smoothcovers} below for a precise statement.

\begin{prop} \label{filtbasis}  Let $\fM_R \in Y^{[a,b], \tau}(R)$.  An eigenbasis $\left \{ \beta^{(j)} \right \}$ for $\fM_R$ induces a trivialization $\fM^{(j)}_{R, \chi} \cong R[\![u]\!]^n$ for any $\chi \in \Delta^*$.  In particular, we have
$$
\gamma^{(j)}: \fM^{(j)}_{R, \chi_{s_j(n)}} \cong R[\![u]\!]^n.
$$   
Similarly, an eigenbasis modulo $P(v)^N$ induces trivializations of $\fM^{(j)}_{R, \chi}$ modulo $E(u)^N$.  
\end{prop}
\begin{proof} An eigenbasis for $\fM_R$ induces a $\Delta$-equivariant trivialization $$\fM^{(j)}_R \cong R[\![v]\!]f^{(j)}_1 \oplus \dots \oplus R[\![v]\!]f^{(j)}_n \cong R[\![v]\!] \otimes_{\La} \tau. $$ We can identify the $\chi$-isotypic component on the right side and see that it is naturally isomorphic to $R[\![u]\!]^n$. To get the explicit basis for the $\chi$-isotypic component, translate the elements of eigenbasis into the $\chi$-isotypic component by multiplying by the smallest non-negative power of $v$ which is compatible with the descent datum. For example, for $\chi_{s_j(n)}$, the basis $\gamma^{(j)}$ will be given by $v^{\bf{a}^{(j)}_{s_j(n)}-\bf{a}^{(j)}_{s_j(1)}}\cdot f^{(j)}_{s_j(1)},\dots, v^{\bf{a}^{(j)}_{s_j(n)}-\bf{a}^{(j)}_{s_j(n-1)}}\cdot f^{(j)}_{s_j(n-1)}, f^{(j)}_{s_j(n)}$. 
\end{proof}

Let $(s_j)_{j \in \Z/f\Z}$ be an orientation of $\tau$ (Definition \ref{orient}).  Furthermore, define a filtration on $\La^n := \tau$ by 
$$
\Fil^k(\La^n) = \sum_{1 \leq i \leq k} (\La^n)_{\chi_{s_j(i)}}.
$$
Let $P_j \subset \GL_n$ be the parabolic which is the stabilizer of $\{ \Fil^k(\La^n) \}$.  For example, if all the characters are distinct then $P_j$ is a Borel subgroup for all $j \in \Z/f\Z$.

\begin{rmk} In Definition \ref{orient}, we define $\lambda^{(j)}_{\bf{a}} \in \Z^n$ which express the characters of $\tau$ in terms of the fundamental character $\omega_f$ in embedding $\sigma_j$.  An orientation is a collection of permutations $(s_j)$ such that $s_j^{-1} (\lambda^{(j)}_{\bf{a}})$ is anti-dominant.  Then $P_j$ is the unique parabolic subgroup containing the diagonal torus and the root groups $U_{\alpha}$ for all $\alpha$ satisfying $\langle \alpha, s_j^{-1} (\lambda^{(j)}_{\bf{a}}) \rangle \leq 0$. In particular, $P_j$ contains the upper triangular Borel $B$. 
\end{rmk}    

Recall the group schemes $L^{+, E_j(u)} \cP_j$ and $\cP_{j,r}$ defined before Proposition \ref{parahoricgroups} with $P_j$ the parabolic as above.  When $p$ is nilpotent in $R$, the $E_j(u)$-adic completion and $u$-adic completions of $R[u]$ coincide and so 
\begin{equation} \label{loopgroup} 
L^{+, E_j(u)} \cP_j (R) = \left \{ g \in GL_n(R[\![u]\!]) \mid g \mod u \in P_j(R) \right \}.
\end{equation}

\begin{prop} \label{smoothcovers} The map $\pi^{(\infty)}:\widetilde{Y}^{[a,b], \tau, (\infty)} \ra Y^{[a,b],  \tau}$ $($resp. $\pi^{(N)}:\widetilde{Y}^{[a,b], \tau} \ra Y^{[a,b],  \tau})$  is a torsor $($for the Zariski topology$)$ for  $\prod_{j \in \Z/f\Z} L^{+, E_j(u)} \cP_j$ $($resp. $\prod_{j \in \Z/f\Z} \cP_{j, N})$.  In particular, $\pi^{(N)}$ is smooth and $\pi^{(\infty)}$ is formally smooth.  
\end{prop} 
\begin{proof}
We observed after Definition \ref{eigenbasis} that an eigenbasis (resp. eigenbasis mod $P(v)^N$) always exists Zariski locally on $\Spec R$.   We focus on the case of $\pi^{(\infty)}$ since the other case is similar.  We want to show that for a given module $\fM_R$ with descent datum of type $\tau$ the set of eigenbases at an embedding $\sigma_j$ is a torsor for $L^{+, E_j(u)} \cP_j(R)$. 

An eigenbasis $\beta^{(j)}:\fM^{(j)} \cong R[\![v]\!]^n$ induces, by taking $\Delta$-invariants, a trivialization $\beta^{(j), \Delta}:\fM^{(j), \Delta = 1} \cong R[\![u]\!]^n$. Thus, for any two eigenbases $\beta^{(j)}$ and $\beta^{' (j)}$, there is 
$
B^{(j)} \in \GL_n(R[\![v]\!]) \text{ such that } \beta^{'(j)}=B^{(j)}\beta^{(j)}, \text{ and } A^{(j)} \in \GL_n(R[\![u]\!]) \text{ such that } \beta^{'(j), \Delta}=A^{(j)}\beta^{(j), \Delta}.
$ 
Concretely, let $\lambda_{\bf{a}}^{(j)} = ({\bf{a}}^{(j)}_i)_i \in \Z^n$ as in Definition \ref{orient}.  Then the relationship between $B^{(j)}$ and $A^{(j)}$ is that 
$$
A^{(j)} = \left(v^{s_j^{-1}(\lambda_{\bf{a}}^{(j)})} \right)^{-1} B^{(j)} \left( v^{s_j^{-1}(\lambda_{\bf{a}}^{(j)})} \right)
$$
where $v^{s_j^{-1}(\lambda_{\bf{a}}^{(j)})}$ is the diagonal matrix with the $(i, i)$th entry given by $v^{\mathbf{a}^{(j)}_{s_j(i)}}$. 

The claim is that  \[\left(v^{s_j^{-1}(\lambda_{\bf{a}}^{(j)})} \right)^{-1}  \GL_n(R[\![v]\!]) \left( v^{s_j^{-1}(\lambda_{\bf{a}}^{(j)})} \right) \cap \GL_n(R[\![u]\!]) = L^{+, E_j(u)} \cP_j(R).\] This can be checked by a direct computation with root groups which we include below.



For the entries below the diagonal, we have
$$
A^{(j)}_{mk} = v^{\mathbf{a}^{(j)}_{s_j(m)} - \mathbf{a}^{(j)}_{{s_j}(k)}} B^{(j)}_{mk}
$$
for $m > k$ and with our choice of ordering $\mathbf{a}^{(j)}_{s_j(m)} -\mathbf{a}^{(j)}_{s_j(k)} \geq 0$ with equality if and only if $\chi_{s_j(m)} = \chi_{s_j(k)}$.  Thus, whenever $\chi_{s_j(m)} \neq \chi_{s_j(k)}$, we see that $A^{(j)}_{mk} \mod u = 0$.  This is exactly the condition $A^{(j)} \mod u \in P_j(R)$.  The converse is also true.  
\end{proof}

  
\begin{thm} \label{PRdescent} For any $r \geq 1$, $Y^{[a,b], \tau}_{r} := Y^{[a,b], \tau} \otimes_{\La} \La/p^r \La$ is representable by an Artin stack of finite type over $\Spec \La/p^r \La$.  Furthermore,  $\widetilde{Y}^{[a,b], \tau}_{r} := \widetilde{Y}^{[a,b], \tau} \otimes_{\La} \La/p^r \La$ is represented by a scheme of finite type over $\Spec \La/p^r \La$. 
\end{thm}
\begin{proof}
It suffices to prove the second statement, for which we will use a strategy originally employed in~\cite{CEGS}. Consider the map
$$
\xi:\widetilde{Y}^{[a,b], \tau}_{r} \to  \widetilde{X}^{[a,b]}_{r}
$$
given by forgetting the descent datum.  It suffices to show that $\xi$ is relatively representable and finite type by Theorem \ref{reprofspace}. 

Given $(\fM_R, \phi_R,\beta_R)\in \widetilde{X}^{[a,b]}_r(R)$ we see that the data of the additive bijections $\widehat{g}:\fM_R\to \fM_R$ for all $g\in \Delta$, which have to commute with $\phi_R$, satisfy $\widehat{g_1\circ g_2}=\widehat{g_1}\circ\widehat{g_2}$, be $R[\![v]\!]$-semilinear and compatible with $\beta_R$ is representable by a scheme of finite type over $R$. Indeed, such a bijection $\widehat{g}$ has to induce an $R(\!(u)\!)$-linear automorphism of $\fM_R[1/v]$ (which can be thought of as an \'etale $\varphi$-module over $R$ of rank $n\cdot (p^f-1)$). By the proof of Theorem 2.5(b) of~\cite{PRcoeff}, the data of an $R(\!(u)\!)$-linear automorphism of $\fM_R[1/v]$ which commutes with $\phi_R$ is representable by a scheme of finite type over $R$. Further imposing the relationships $\widehat{g_1\circ g_2}=\widehat{g_1}\circ\widehat{g_2}$ and the $R[\![v]\!]$-semilinearity cuts out a closed subscheme. Finally, the requirement that the descent datum preserve the lattice $\fM_R\subset \fM_R[1/v]$ and the compatibility with $\beta_R$ are also closed conditions.

We conclude then that $\widetilde{Y}^{[a,b], \tau}_{r} \ra \widetilde{X}^{[a,b]}_r$ is relatively representable and finite type and so by Theorem \ref{reprofspace} $\widetilde{Y}^{[a,b], \tau}_{r}$  is a scheme of finite type over $\Spec \La/p^r \La$.   Since $\widetilde{Y}^{[a,b], \tau}_{r} \ra Y^{[a,b], \tau}_{r}$ is a smooth cover we deduce that $Y^{[a,b], \tau}_{r}$ is an Artin stack of finite type.  
\end{proof} 


We are now ready to construct the local model diagram for Kisin modules with descent data:
$$
\xymatrix{
& \widetilde{Y}^{[0,h], \tau, (\infty)} \ar[dl]_{\pi^{(\infty)}} \ar[dr]^{\Psi} & \\
Y^{[0,h], \tau} & & \Fl_{K}^{E(u)}. \\
}
$$

\noindent To define $\Psi$, we need to associate to any $(\fM_R, \phi_{R}, \{\widehat{g}\}, \beta_R) \in \widetilde{Y}^{[a,b], \tau, (\infty)}(R)$ and each embedding $\sigma_j$, a triple $(L^{(j)}, \alpha^{(j)}, \eps^{(j)}) \in \Fl^{E_j(u)}_{P_j}(R)$.   The pair $(L^{(j)}, \alpha^{(j)})$ is straightforward to define and is given by the `image' of Frobenius.  
 
 To be precise, we take $L^{(j)} = ^{\phz} \fM^{(j-1)}_{R, \chi_{s_j(n)}}$ and define the trivialization $\alpha^{(j)}$ by the composition 
 \begin{equation} \label{a1}
 ^{\phz} \fM^{(j-1)}_{R, \chi_{s_j(n)}}[1/E_j(u)] \xrightarrow{\phi^{(j-1)}_{R, \chi_{s_j(n)}}} \fM^{(j)}_{R, \chi_{s_j(n)}}[1/E_j(u)] \xrightarrow{\gamma^{(j)}} (R[\![u]\!]^n)[1/E_j(u)]
 \end{equation}
 where $\gamma^{(j)}$ is induced by $\beta^{(j)}$ as in Proposition \ref{filtbasis}.  Notice that we are using the alternative description of $\Fl_{P_j}^{E_j(u)}$ from Definition \ref{alt}.  
 
 Next, we have to define a filtration $\eps^{(j)}$ on $L^{(j)} \mod u$.   Let 
 $$
 ^{\phz} D^{(j-1)}_{\chi_{s_j(n)}} := ^{\phz} \fM^{(j-1)}_{R, \chi_{s_j(n)}} \mod u = L^{(j)} \mod u.
 $$  
 The filtration is essentially given by the diagram (\ref{bigdiagram}).   Namely for each $1 \leq i \leq n$, let 
 $$
 \omega_i: ^{\phz} \fM^{(j-1)}_{R, \chi_{s_j(i)}} \ra ^{\phz} \fM^{(j-1)}_{R, \chi_{s_j(n)}}
 $$ 
 be the injective map induced by composition along the upper row of (\ref{bigdiagram}).  Then we get the inclusions
 $$
 u \left(^{\phz} \fM^{(j-1)}_{R, \chi_{s_j(n)}} \right) \subset \omega_1 \left(^{\phz} \fM^{(j-1)}_{R, \chi_{s_j(1)}} \right) \subset \ldots \omega_{n-1} \left( ^{\phz} \fM^{(j-1)}_{R, \chi_{s_j(n-1)}} \right) \subset ^{\phz} \fM^{(j-1)}_{R, \chi_{s_j(n)}}.
 $$
We can then define the filtration $\eps^{(j)}$ by  
\begin{equation} \label{descfilt}
\Fil^i  \left({}^{\phz} D^{(j-1)}_{\chi_{s_j(n)}} \right) = \omega_i \left( ^{\phz} \fM^{(j-1)}_{R, \chi_{s_j(i)}} \right) /  u \left (^{\phz} \fM^{(j-1)}_{R, \chi_{s_j(n)}} \right).
\end{equation}
It is not hard to see that the filtration $\eps^{(j)}$  is a  $P_j$-filtration for $P_j$ defined after Proposition \ref{filtbasis}.  

In summary, we have
$$\Psi (\fM_R, \phi_{R}, \{ \widehat{g} \}, \beta^{(j)}) = \left(^{\phz} \fM^{(j-1)}_{R, \chi_{s_j(n)}}, \gamma^{(j)} \circ \phi^{(j -1)}_{R, \chi_{s_j(n)}}, \left\{ \mathrm{Fil}^i \left({}^{\phz} D^{(j-1)}_{\chi_{s_j(n)}}\right) \right\}_{i=1}^{n} \right)_{j \in \Z/f\Z}.
$$

\begin{exam} 
Let $K=K_0=\Q_p$ (so $f=1$) and $L$ be the ramified extension of $\Qp$ of degree $p-1$, such that $E(u)=u+p$ and $P(v)=v^{p-1}+p$. Let $0\leq a<b<p-1$ and consider the two dimensional tame type over $\Qp$ given by $\tau_1=\omega_1^a\oplus \omega_1^b$. This is a continuation of Example~\ref{cuspidal}.

We define $\fM_{\F}$ to be the rank $2$ Kisin module over $\F$ given by $\fM_{\F}=e_1\mathbb{F}[\![v]\!]\oplus e_2\mathbb{F}[\![v]\!]$ with $\phi$ given by:
$$
\phi(e_1)=e_1, \phi(e_2)=v^{p-1}e_2$$ 
and with descent datum $\tau_1$ given by:
$$
g(e_1)=\omega_1^a(g)e_1, g(e_2)=\omega_1^b(g)e_2
$$
for every $g\in \mathrm{Gal}(L/\Qp)$. 

Then $(e_1,e_2)$ is an eigenbasis for $\fM_\F$ and induces the bases $(e_1,v^{p-1+a-b}e_2)$ of $^{\phz} \fM_{\F,\omega_1^a}$ and $(v^{b-a}e_1,e_2)$ of $^{\phz} \fM_{\F,\omega_1^b}$ as $\F[\![u]\!]$-modules, as in Proposition~\ref{filtbasis}. The matrix of the trivialization of $^{\phz} \fM_{\F,\omega_1^b}$ is $\left(\begin{smallmatrix}1&0\\0&u\end{smallmatrix}\right)$. The map $\omega_1:^{\phz} \fM_{\F,\omega_1^a}\to ^{\phz} \fM_{\F,\omega_1^b}$ is multiplication by $v^{b-a}$. Therefore, the induced filtration on $^\phz D_{\omega^b_1}$ has $\mathrm{Fil}^1$ generated by the image of the basis element $v^{b-a}e_1$ modulo $u$. ($\mathrm{Fil}^0$ is everything and $\mathrm{Fil}^{2}=\{0\}$.)  

We leave the case of the base change to $\mathbb{Q}_{p^2}$ of the cuspidal type $\tau_2$ as an exercise. 
\end{exam}

\noindent We now come to the main theorem:
\begin{thm}\label{formal-smoothness of psi} The morphism $\Psi$ is formally smooth.
\end{thm}  
\begin{proof} Roughly, the idea is that the image under $\Psi$ gives the descent datum and the image of Frobenius. What is left is to choose an isomorphism between $\phz^*(\fM_R)$ and its image which is compatible with descent datum. By using the diagram (\ref{bigdiagram}), we show that it is enough to choose the isomorphism on the $\chi_{s_j(n)}$th isotypic pieces and that any choice works, giving formal smoothness. We now give the details. 

We can twist to reduce the case where $[a,b] = [0,h]$ so that the Frobenius is an honest endomorphism of $\fM_R$.  Let $R \in \Nilp_{\La}$ and let $I$ be a square-zero ideal of $R$.  Choose $(\fM_{R/I}, \phi_{R/I}, \{ \widehat{g} \}, \beta^{(i)}) \in  \widetilde{Y}^{[0,h], \tau, (\infty)}(R/I)$.  Assume we are given a lift $(L^{(j)}_R, \widetilde{\alpha}^{(j)}, \{ \mathrm{Fil}^i(L^{(j)} \mod v) \})$  of $\Psi(\fM_{R/I})$ to $R$.

Let $\fM_R$ be a free $(W \otimes_{\Zp} R)[[v]]$-module of rank $n$ and choose an isomorphism $\fM_R \otimes_R R/I \cong \fM_{R/I}$. By Proposition  \ref{filtbasis}, $\beta^{(j)}_{R/I}$ induces a trivialization $\gamma^{(j)}_{R/I}:\fM^{(j)}_{R/I, \chi_{s_j(n)}} \cong (R/I)[\![u]\!]^n$.   We can then choose trivializations $\widetilde{\beta}^{(j)}$ of $\fM_R^{(j)}$ for each $j$ such that the diagram
$$
\xymatrix{
\fM^{(j)}_R \ar[r]^{\widetilde{\beta}^{(j)}} \ar[d] & R[[v]]^n \ar[d] \\
\fM^{(j)}_{R/I} \ar[r]^{\beta^{(j)}} & (R/I) [[v]]^n \\
}
$$
commutes. Let $f^{(j)}_{s_j(i)}$ be the preimage of the $i$th standard basis element under $\widetilde{\beta}^{(j)}$.  We define a semilinear $\Delta$-action of type $\tau$ on $\fM_R$ by demanding that $\Delta$ act on $f^{(j)}_{s_j(i)}$ through the character $\chi_{s_j(i)}$.  This clearly makes $\widetilde{\beta}^{(j)}$ into an eigenbasis for this descent datum.

The eigenbasis $\widetilde{\beta}^{(j)}$ induces a filtration on $^{\phz} D^{(j)}_{\chi_{s_j(n)}} = \left(^{\phz} \fM^{(j)}_{R, \chi_{s_j(n)}}\right)/u\left(^{\phz} \fM^{(j)}_{R, \chi_{s_j(n)}}\right)$ as in (\ref{descfilt}) (compatible with reduction modulo $I$). Choose an isomorphism $\widetilde{\theta}^{(j)}:^{\phz} \fM^{(j-1)}_{R, \chi_{s_j(n)}} \cong L^{(j)}_R$ compatible with the filtrations on $^{\phz} D^{(j)}_{\chi_{s_j(n)}}$ and $L^{(j)}_R/ u L^{(j)}_R$ and such that the diagram
$$
\xymatrix{
\fM^{(j-1)}_{R, \chi_{s_j(n)}}\ar[r]^{\widetilde{\theta}^{(j)}} \ar[d] &  L^{(j)}_R \ar[d] \\
^{\phz} \fM^{(j-1))}_{R/I, \chi_{s_j(n)}}  \ar[r]^{\theta^{(j)}} &L^{(j)}_{R/I} \\
}
$$
commutes. (The isomorphism $\theta^{(j)}$ is already compatible with the filtrations on the reductions modulo $u$ by our assumption that $\Psi(\fM_{R/I})$ equals $L^{(j)}_{R/I}$ together with the induced extra data.)

Define $\phi^{(j-1)}_{R, \chi_{s_j(n)}}$ to be the composition
$$
^{\phz} \fM^{(j-1)}_{R, \chi_{s_j(n)}}[1/E_j(u)] \xrightarrow{\widetilde{\theta}^{(j)}} L^{(j)}_R[1/E_j(u)] \xrightarrow{\widetilde{\alpha}^{(j)}} (R[[u]]^n)[1/E_j(u)] \xrightarrow{(\widetilde{\gamma}^{(j)})^{-1}} \fM^{(j)}_{R, \chi_{s_j(n)}}[1/E_j(u)].
$$
Observe that the only map which not an isomorphism without inverting $E_j(u)$ is $\widetilde{\alpha}^{(j)}$.  The ``image'' of Frobenius is then determined by the image of $\widetilde{\alpha}^{(j)}$. 

If $\fM_R \in \widetilde{Y}^{[0,h], \tau, (\infty)}(R)$ then the Frobenius $\phi^{(j-1)}_R$ is uniquely determined by  $\phi^{(j-1)}_{R, \chi_{s_j(n)}}$ by diagram (\ref{bigdiagram}). Indeed, to construct $\phi^{(j-1)}_R$ it suffices to construct $\phi^{(j-1)}_{R, \chi_{s_j(i)}}$ for each $1 \leq i \leq n -1$ such that the diagram
$$
\xymatrix{
^{\phz} \fM^{(j-1)}_{R, \chi_{s_j(i)}} \ar[r]^{\omega_i} \ar[d]^{\phi^{(j-1)}_{R, \chi_{s_j(i)}}} & ^{\phz} \fM^{(j-1)}_{R, \chi_{s_j(n)}} \ar[d]^{\phi^{(j-1)}_{R, \chi_{s_j(n)}}} \\
\fM^{(j)}_{R, \chi_{s_j(i)}} \ar[r]^{\omega'_i} &  \fM^{(j)}_{R, \chi_{s_j(n)}} \\
}
$$
commutes. The horizontal arrows are induced by multiplication by $v^{\bf{a}^{(j)}_{s_j(n)}-\bf{a}^{(j)}_{s_j(i)}}$, so they are injections by Proposition~\ref{isopieces}.  The fact that $\widetilde{\theta}^{(j)}$ was chosen to respect filtrations implies that the composition $\phi^{(j-1)}_{R, \chi_{s_j(n)}} \circ \omega_i$ lies in the image of the bottom horizontal map $\omega_i'$ and so there exists a unique $\phi^{(j-1)}_{R, \chi_{s_j(i)}}$ which completes the diagram.     
\end{proof}

We can refine $\Psi$ to a morphism of finite type.  
 \begin{prop} Let $N > b-a$.  The map $\Psi$ factors through the finite type closed subscheme $\Fl^{[a,b], E(u)}_{K}$.  Furthermore, there exists a smooth map $\Psi^N: \widetilde{Y}^{[a,b], \tau} \ra \Fl^{[a,b], E(u)}_{K}$ such that $\Psi$ is the composition of 
 $$
\widetilde{Y}^{[a,b], \tau, (\infty)} \ra \widetilde{Y}^{[a,b], \tau} \xrightarrow{\Psi^N} \Fl^{[a,b], E(u)}_{K}.
$$    
\end{prop}
\begin{proof}
Lemma \ref{vheight} says that image of $\Psi$ factors through $\Fl^{[a,b], E_j(u)}_{P_j}$ on each factor and hence through $\Fl^{[a,b], E(u)}_{K}$.  

To show that $\Psi$ factors as $\Psi^N$, we have to show that for any $(\fM_R, \phi_{R}, \{ \widehat{g} \}, \beta) \in \widetilde{Y}^{[a,b], \tau, (\infty)}(R)$ the image under $\Psi$ only depends on $\beta$ modulo $P(v)^N$. The image under $\Psi$ is the tuple
$$ 
\left(^{\phz} \fM^{(j-1)}_{R, \chi_{s_j(n)}}, \gamma^{(j)} \circ \phi^{(j -1)}_{R, \chi_{s_j(n)}}, \left\{ \mathrm{Fil}^i \left({}^{\phz} D^{(j-1)}_{\chi_{s_j(n)}}\right) \right\}_{i=1}^{n} \right)_{j \in \Z/f\Z}.
$$
From the construction of the map $\Psi$, we see that the eigenbases $\beta^{(j)}$ only affect the construction of trivializations $\gamma^{(j)}$ in (\ref{a1}). The eigenbasis $\beta^{(j)}$ has no effect on $^{\phz} \fM^{(j-1)}_{R, \chi_{s_j(n)}}$ or on the filtration on $^{\phz} D^{(j-1)}_{\chi_{s_j(n)}}$.  Furthermore, as we saw in the proof of Proposition \ref{smoothcovers}, changing the eigenbasis $\beta^{(j)}$ amounts to composing $\gamma^{(j)}: \fM^{(j)}_{R, \chi_{s_j(n)}} \cong (R[\![u]\!])^n$ with an element of $g \in L^{+, E_j(u)} \cP_j(R)$.  On $\Fl^{[a,b], E_j(u)}_{P_j}$, this corresponds to the natural left action of $L^{+, E_j(u)} \cP_j$ defined in  Proposition \ref{niceaction}.  

If $\beta^{(j)}$ and $\beta^{'(j)}$ are congruent modulo $\sigma_j(P(v))^N$, then $\gamma^{(j)} = g\cdot\gamma^{'(j)}$ for $g \in L^{+, E_j(u)} \cP_j (R)$ with $g \equiv \mathrm{Id} \mod E_j(u)^N$.   If $g$ is congruent to the identity modulo $E_j(u)^N$, then $g$ acts trivially on $\Fl^{[a,b], E_j(u)}_{P_j}$ (for example, by identifying $\Fl^{[a,b], E_j(u)}_{P_j}$ with lattices as in Definition \ref{BDlattice}).   
\end{proof} 

\begin{cor} There is a diagram 
$$
\xymatrix{
& \widetilde{Y}^{[a,b], \tau} \ar[dl]_{\pi^{N}} \ar[dr]^{\Psi^N} & \\
Y^{[a,b], \tau} & & \Fl_K^{[a,b], E(u)}, \\
}
$$
where both $\pi^N$ and $\Psi^N$ are smooth. 
\end{cor}

\section{$p$-adic Hodge type}

In this section, we define and study a closed substack $Y^{\mu, \tau} \subset Y^{[a,b], \tau}$ which is related to the notion of $p$-adic Hodge type. A similar construction but without descent data was carried out in \cite[\S 3]{PRcoeff}.  When $n = 2$ and $\mu \in (\{0,1\}^n)^{\Hom(K, \overline{\Q}_p)} $ (i.e., $\mu$ minuscule), $Y^{\mu, \tau}$ and the local model diagram are studied in forthcoming work of the first author with Emerton, Gee and Savitt~\cite{CEGS}. 

Let $\mu$ be a geometric cocharacter of $\Res_{K/\Qp} \GL_n$.   For each embedding $\sigma_j:K_0 \ra F$, we get a geometric cocharacter $\mu_j$ of $\Res_{K/K_0} \GL_n$ such that $\mu = (\mu_j)_{\sigma_j}$. Assume that $F = \La[1/p]$ contains the reflex field of the conjugacy class $[\mu]$, i.e., $\La = \La_{[\mu]}$. 

In \S 2, we defined the local model
$$
M(\mu) = \prod_{j \in \Z/f\Z} M(\mu_j) \subset  \prod_{j \in \Z/f\Z} \Fl^{E_j(u)}_{P_j} = \Fl^{E(u)}_K.
$$ 
By Theorem \ref{locmodels}, $M(\mu)$ is flat and projective over $\La$ with reduced special fiber.   Also, $M(\mu)$ is stable for the action of the ``loop group'' $\prod_{j \in \Z/f\Z} L^{+, E_j(u)} \cP_j$ by Proposition \ref{niceaction}. 

Assume that $a, b$ are integers with $a \leq b$ such that $M(\mu) \subset \prod_{i \in \Z/f\Z} \Fl^{[a,b], E_j(u)}_{P_j}$.  For any $N > a-b$, we saw in Proposition~\ref{niceaction} that the action of $\prod_{j \in \Z/f\Z} L^{+, E_j(u)} \cP_j$ on $M(\mu)$ factors through the action of the smooth connected group scheme $\prod_{j \in \Z/f\Z} \cP_{j, N}$.  

\begin{defn} Define the closed subscheme
$$
\widetilde{Y}^{\mu, \tau} := \widetilde{Y}^{[a,b], \tau} \times_{\Fl^{[a,b], E(u)}_K, \Psi^N} M(\mu).
$$ 
\end{defn} 

We have an induced smooth map
$$
\Psi^{\mu}: \widetilde{Y}^{\mu, \tau} \ra M(\mu).  
$$
We would like to show that $\widetilde{Y}^{\mu, \tau}$ descends to a closed substack $Y^{\mu, \tau} \subset Y^{[a,b], \tau}$.

\begin{prop} \label{mudescent}  For any $r \geq 1$, there exists a closed substack $Y^{\mu, \tau}_r \subset Y^{[a,b], \tau}_r$ such that the diagram
$$
\xymatrix{
\widetilde{Y}^{\mu, \tau}_r \ar[d]_{\pi^{\mu}} \ar[r] & \widetilde{Y}^{[a,b], \tau}_r \ar[d]^{\pi^{(N)}}\\
Y^{\mu, \tau}_r \ar[r] & Y^{[a,b], \tau}_r  \\
}
$$ 
is Cartesian. Furthermore, $Y^{\mu, \tau}_r \times_{\Z/p^r \Z} \Z/p^{r-1} \Z \cong Y^{\mu, \tau}_{r-1}$.
\end{prop}
\begin{proof} 
By Proposition \ref{smoothcovers}, $\pi^{(N)}:\widetilde{Y}^{[a,b], \tau}_r \ra Y^{[a,b], \tau}_r$ is a torsor for the smooth group $\cG_r := (\prod_{j \in \Z/f\Z} \cP_{j, N})_{\La/p^r \La}$.   Any $\cG_r$-stable closed subscheme of $\widetilde{Y}^{[a,b], \tau}_r$ descends by faithfully flat descent to a closed substack of  $Y^{[a,b], \tau}_r$.  

Since $(M(\mu))_{\La/p^r \La}$ is stable under $\cG_r$ so is  $\widetilde{Y}^{\mu, \tau}_r$ and we define the desired $Y^{\mu, \tau}_r$ by descent.   This construction is clearly compatible with reduction modulo $p^{r-1}$.     
\end{proof} 

\noindent Since the $Y^{\mu, \tau}_r$ are compatible with reduction modulo $p^{r-1}$, we can define a stack $Y^{\mu, \tau}$ on $\Nilp_{\La}$ whose reduction modulo $p^r$ is $Y^{\mu, \tau}_r$.  


\begin{thm}\label{main thm in body} We have a local model diagram:  
\begin{equation} \label{lmdiagram}
\xymatrix{
& \widetilde{Y}^{\mu, \tau} \ar[dl]_{\pi^{\mu}} \ar[dr]^{\Psi^{\mu}} & \\
Y^{\mu, \tau} & & M(\mu) \\
}
\end{equation}
where both $\pi^{\mu}$ and $\Psi^{\mu}$ are smooth maps. 
\end{thm}

\subsection{Special fiber: Kottwitz-Rapoport strata} 

In addition to imposing the $p$-adic Hodge type $\mu$ via the local model diagram (\ref{lmdiagram}), we can also stratify the special fiber of $Y^{\mu, \tau}$ by pulling back the stratification in Theorem~\ref{lm special}.  This is the analogue of the Kottwitz-Rapoport stratification in the Shimura variety setting.  

Let $\overline{Y}^{\mu, \tau}$ denote the special fiber of $Y^{\mu, \tau}$.  As in the discussion before Theorem~\ref{lm special}, we can write $\mu_j = (\mu_{j, \psi})$ where $\psi$ runs over embeddings $\psi:K \iarrow \overline{F}$ that extend $\sigma_j$ and where each $\mu_{j, \psi}$ is dominant. We define 
$$
\lambda_j = \sum_{\psi:K \iarrow \overline{F}} \mu_{j, \psi}. 
$$

\begin{prop} For each $\widetilde{w} = (\widetilde{w}_j) \in \prod_{j =0}^{f-1} \Adm_{P_j} (\lambda_j)$, there is a locally closed substack $\overline{Y}^{\mu, \tau}_{\widetilde{w}} \subset \overline{Y}^{\mu, \tau}$ such that 
$$
(\pi^{\mu})^{-1} \left(\overline{Y}^{\mu, \tau}_{\widetilde{w}} \right) =  (\Psi^{\mu})^{-1} \left(\prod_{j} S^0(\widetilde{w}_j)\right).
$$ 
Furthermore, the closure $\overline{Y}^{\mu, \tau}_{\leq \widetilde{w}}$ of $\overline{Y}^{\mu, \tau}_{\widetilde{w}}$ is the union of the strata for all $(\widetilde{w}'_j) \in \prod_{j =0}^{f-1} \Adm_{P_j} (\lambda_j)$ such that $\widetilde{w}'_j \leq \widetilde{w}_j$ for all $j$. 
\end{prop} 
\begin{proof} 
We would like to define $\overline{Y}^{\mu, \tau}_{\widetilde{w}}$ by faithfully flat descent from $(\Psi^{\mu})^{-1} \left(\prod_{j} S^0(\widetilde{w}_j)\right)$.  To descend along $\pi^{\mu}$, we need $(\Psi^{\mu})^{-1} \left(\prod_{j} S^0(\widetilde{w}_j)\right)$ or equivalently $\prod_{j} S^0(\widetilde{w}_j)$  to be stable under the action of $(L^{+, E_j(u)} \cP_j)_{\F}$.  The group scheme $(L^{+, E_j(u)} \cP_j)_{\F}$  is the parahoric group scheme $\cP_j$ corresponding to $P_j$ defined before Theorem \ref{lm special} whose orbits are exactly the open affine Schubert cells $S^0(\widetilde{w}_j)$.  Since $\overline{M}(\mu)$ is union of $\prod_{j} S^0(\widetilde{w}_j)$ for $\widetilde{w}_j \in \Adm_{P_j}(\lambda_j)$ (Theorem \ref{lm special}), the same is true for $\overline{Y}^{\mu, \tau}$.  The closure relations follow from smoothness of $\pi^{\mu}$ and $\Psi^{\mu}$.   
\end{proof}

We now introduce the notion of \emph{shape} (or \emph{genre} in French).  The \emph{genre} of Kisin/Breuil module of rank 2 was first introduced in~\cite{BreuilLG} where it is connected to Serre weights for $\GL_2$ over an unramified extension of $\mathbb{Q}_p$. It also plays an important role in~\cite{BM2, EGS} in computing tamely Barsotti-Tate deformation rings as well as in the recent work of~\cite{CDM1, CDM2}. The notion of shape for a rank 3 Kisin modules with $p$-adic Hodge type $(2,1,0)$ and $K/\Qp$ unramified will be used in forthcoming joint work of the second author~\cite{LLLM1} to compute potentially crystalline deformation rings for $\GL_3$. 

\begin{defn} \label{shape}
A Kisin module $\overline{\fM} \in \overline{Y}_{\widetilde{w}}^{\mu, \tau}(\overline{\F}_p)$ is said to have \emph{shape} $\widetilde{w}$. 
\end{defn}

\begin{rmk} The shape of Kisin module $\overline{\fM} \in \overline{Y}^{\mu, \tau}(\overline{\F})$ has a more concrete interpretation as well. $\overline{\fM}$ has shape $(\widetilde{w}_j)$ if the matrix for the Frobenius $\phi^{(j)}_{\overline{\F}, \chi_{s_j(n)}}$ with respect to any basis compatible with the filtration lies in the double coset $L^+ \cP_j(\overline{\F}) \widetilde{w}_j L^+ \cP_j(\overline{\F})$. 
\end{rmk}
  
\subsection{Generic fiber} 

We would now like to characterize $Y^{\mu, \tau}$ so that we can relate it back to potentially crystalline representations and Hodge-Tate weights in the next section.  Since $M(\mu)$ is defined by flat closure, this has to be done by working over the ``generic'' fiber in some suitable sense.  


For any complete local Noetherian $\La$-algebra $R$ with finite residue field and maximal ideal $m_R$, we define the $R$-points of $Y^{[a,b],  \tau}$ as the inverse limit category
$$
Y^{[a,b],  \tau}(R) = \{   (\fM_k, \iota_k) \mid \fM_k \in Y^{[a,b],  \tau}(R/m_R^k R) , \iota_{k} : \fM_k \otimes R/m_R^{k-1} R \cong \fM_{k-1} \}.
$$
Similarly, we can define $Y^{\mu, \tau}(R)$.

Given $(\fM_k, \iota_k) \in Y^{[a,b],  \tau}(R)$, the inverse limit $\fM_R = \varprojlim \fM_k$ is a module over $(W \otimes_{\Zp} R)[\![v]\!]$ equipped with a semilinear Frobenius 
$$
\phi_R:\phz^*(\fM_R)[1/P(v)] \ra \fM_R[1/P(v)]
$$
and descent datum of type $\tau$. 

We now introduce the notion $p$-adic Hodge type first for $\overline{\Q}_p$-points and then more generally.   Let $F'/F$ be a finite extension with ring of integers $\La'$. 

\begin{prop} \label{modPvred} For any Kisin module $\fM_{\La'} \in Y^{[a,b],  \tau}(\La')$, let $\fM_{F'} := \fM_{\La'}[1/p]$. Then the specialization
$$
\cD_{F'} := \phz^*(\fM_{F'})/ P(v) \phz^*(\fM_{F'})
$$
is a finitely generated projective $L \otimes_{\Qp} F'$-module with a semilinear action of $\Delta$.  
\end{prop}
\begin{proof}  This follows from the fact that $((W \otimes_{\Zp} \La')[\![v]\!])[1/p]/P(v) \cong L \otimes_{\Qp} F'$ and that $\fM_{F'}$ is finitely generated and  projective over $((W \otimes_{\Zp} \La')[\![v]\!])[1/p]$.
\end{proof} 

We can define a filtration on $\cD_{F'}$ as in \cite{KisinPSS}.
\begin{defn} \label{KisinHodge} Define
$$
\Fil^i( \phz^*(\fM_{F'})) := \{ m \in \phz^*(\fM_{F'}) \mid \phi_{\fM_{F'}}(m) \in P(v)^i \fM_{F'} \}. 
$$
Define $L \otimes_{\Qp} F'$-submodules
$$ 
\Fil^i(\cD_{F'}) := \Fil^i( \phz^*(\fM_{F'})) / (\Fil^i( \phz^*(\fM_{F'}))  \cap P(v)  \phz^*(\fM_{F'}))  \subset \cD_{F'}.
$$
\end{defn} 

\begin{rmk} If $\fM_{F'}$ has height in $[a,b]$ then it is a decreasing filtration with $\Fil^a(\cD_{F'}) = \cD_{F'}$ and $\Fil^{b+1}(\cD_{F'}) = 0.$  
\end{rmk}

For $\fM_{F'}$ as in Proposition \ref{modPvred} and $\chi\in \Delta^*$, we can define $\cD^{(j)}_{F', \chi} := ^{\phz} \fM^{(j-1)}_{F', \chi}/ E_j(u) ^{\phz} \fM^{(j-1)}_{F', \chi}$ together with a filtration defined in an analogous way using $\phi^{(j-1)}_{\fM_{F'}, \chi}$ and $E_j(u)$ in place of $\phi_{\fM_{F'}}$ and $P(v)$. 

\begin{lemma} \label{modEured} Let $\fM_{F'}$ be as in Proposition \ref{modPvred}. Let $\cD_{F'}^{(j)}$ be the $L\otimes_{K_0, \sigma_j}F'$-submodule of $\cD_{F'}$ corresponding to $\sigma_j:K_0\hookrightarrow F'$. There is a natural isomorphism 
$$
\cD_{F'}^{(j)} \cong \cD^{(j)}_{F', \chi_{s_j(n)}} \otimes_K L
$$
of filtered $L \otimes_{K_0, \sigma_j}F'$-modules.  
\end{lemma}
\begin{proof} First, note that we have the isotypic decomposition $\cD_{F'}^{(j)}=\oplus_{\chi\in \Delta^*} \cD^{(j)}_{F',\chi}$ as $K\otimes_{K_0,\sigma_j}F'$-modules, which gives an isomorphism $\cD_{F'}^{(j)} \cong \cD^{(j)}_{F', \chi_{s_j(n)}} \otimes_K L$ of $K\otimes_{K_0,\sigma_j}F'$-modules, since multiplication by $v$ when $p$ is inverted and $P(v)=0$ induces isomorphisms $\cD^{(j)}_{F',\chi}\cong \cD^{(j)}_{F', (\sigma_j\circ\omega_f)\chi}$ as $F'$-modules. This can be upgraded to an isomorphism $\cD_{F'}^{(j)} \cong \cD^{(j)}_{F', \chi_{s_j(n)}} \otimes_K L$ of $L\otimes_{K_0,\sigma_j}F'$-modules, because multiplication by $v^{p^{f}-1}$ is multiplication by $u$, which is identified with $\pi_K\otimes 1$ under the isomorphism $((W \otimes_{K_0,\sigma_j} \La')[\![u]\!])[1/p]/E_j(u) \cong K \otimes_{K_0,\sigma_j} F'$. This means that $v$ is identified with $\pi_L\otimes1 \in L\otimes_{K_0,\sigma_j}F'$. The fact that the isomorphism $\cD_{F'}^{(j)} \cong \cD^{(j)}_{F', \chi_{s_j(n)}} \otimes_K L$ respects the filtrations on the two sides follows from the commutative diagram~\ref{bigdiagram}, where all the horizontal maps are now isomorphisms. 
\end{proof}

Recall that we assume the conjugacy class of $\mu$ is defined over $F$, i.e., $F = F_{[\mu]}$. Associated to $\mu$, we then have a $\Z$-graded $K \otimes_{\Qp} F$-module $V_{\mu}$ of rank $n$.  See for example \cite[(2.6)]{KisinPSS}.  
 
 \begin{defn} Let $F'/F$ be a finite extension with ring of integers $\La'$.  We say that $\fM_{\La'} \in  Y^{[a,b],  \tau}(\La')$ has \emph{$p$-adic Hodge type $\mu$} if 
 $$
 \gr^{\bullet}(\cD_{F'}) \cong \gr^{\bullet} (V_{\mu} \otimes_{K \otimes_{\Qp} F} (L \otimes_{\Qp} F'))
 $$ 
 as graded $L \otimes_{\Qp} F'$-modules.  
 
We say $\fM_{\La'}$ has \emph{$p$-adic Hodge type $\leq \mu$} if $\fM_{\La'}$ has $p$-adic Hodge type $\mu'$, for some $\mu'$ such that $[\mu'] \leq [\mu]$ in the Bruhat ordering.  
 \end{defn}

\begin{cor}\label{LoverK} Let $F'/F$ be a finite extension and let $\fM_{\La'}$ be as above. Write $\mu =(\mu_j)_{j\in \Z/f\Z}$, where each $\mu_j$ is a geometric cocharacter of $\mathrm{Res}_{K/K_0}GL_n$. Given $\mu_j$, let $V_{\mu_j}$ be the filtered $K \otimes_{K_0, \sigma_j} F$-module of rank $n$ corresponding to it as above. Then $\fM_{\La'}$ has $p$-adic Hodge type $\mu = (\mu_j)_{j \in \Z/f\Z}$ if and only if 
$$
\gr^{\bullet}(\cD_{F', \chi_{s_j(n)}}) \cong \gr^{\bullet} (V_{\mu_j}) 
$$
for every $0\leq j\leq f-1$.
\end{cor} 
\begin{proof} This follows directly from Lemma \ref{modEured}. 
\end{proof} 
 
Let $\fM_R \in Y^{[a,b],  \tau}(R)$.  For any finite extension $F'/F$, any homomorphism $x:R \ra F'$ factors through the ring of integers $\La'$.  We can consider the base change $\fM_x := (\fM_R \otimes_{R, x} \La')[1/p]$ for which we have defined the notion of $p$-adic Hodge type. 

We would now like to characterize when $\fM_R \in Y^{[a,b], \tau}(R)$ lies in $Y^{\mu, \tau}(R)$.   
 
\begin{thm}\label{p-adic Hodge type} Let $R$ be a complete local Noetherian $\La$-algebra with finite residue field.  Assume $R$ is $\La$-flat and reduced.  Then $\fM_R \in Y^{[a,b], \tau}(R)$ lies in $Y^{\mu, \tau}$ if and only if for all finite extensions $F'/F$ and all homomorphisms $x:R \ra F'$ the base change $\fM_x$ has $p$-adic Hodge type $\leq \mu$.
\end{thm} 
\begin{proof}
Let $N>a-b$. Choose an eigenbasis $\widetilde{z}_1:=\left(\overline{\beta}^{(j)} \right)_{j \in \Z/f\Z}$ for $\fM_{R, 1} \in Y^{[a,b], \tau}(R/m_R)$.  Since the morphism $\pi^{(N)}: \widetilde{Y}^{[a,b], \tau}\to Y^{[a,b], \tau}$ is smooth, we can find a compatible system of points $\widetilde{z}:=\left(\widetilde{z}_r\right)_r$ with $\widetilde{z}_r \in \widetilde{Y}^{[a,b], \tau, (N)}(R/m_R^r)$ such that $\pi^{(N)}(\widetilde{z}_r) = \fM_{R, r}$.     

We see then that $\fM_R$ is in $Y^{\mu, \tau}(R)$ if and only if $\Psi^N(\widetilde{z}_r) \in M(\mu)(R/m_R^r)$ for all $r \geq 1$.  The compatible system $\Psi^N(\widetilde{z}_r)$ defines a map
$$
\Psi^N(\widetilde{z}): \Spec R \ra \Fl_K^{[a,b], E(u)}.  
$$
Since $M(\mu)$ is a $\La$-flat closed subscheme of $\Fl_K^{[a,b], E(u)}$, we see that $\Psi^N(\widetilde{z})$ factors through $M(\mu)$ if and only if we have a factorization
$$
\xymatrix{
\Spec R[1/p] \ar@{.>}[rd] \ar[r]^{\Psi^N(\widetilde{z})[1/p]} & \Fl_K^{[a,b], E(u)} \\
 & M(\mu)[1/p] = \prod_j (1_{\GL_n/P_j} \times S(\mu_j)). \ar@{^{(}->}[u]  \\
}
$$ 
Since $R[1/p]$ is reduced and Jacobson, it suffices to show that we have a factorization at the level of $\overline{\Q}_p$-points. 

Every eigenbasis modulo $P(v)^N$ lifts to an eigenbasis, so we can switch from considering $\Psi^N:\widetilde{Y}^{[a,b],\tau} \to \Fl_K^{[a,b], E(u)}$ to considering $\Psi: \widetilde{Y}^{[a,b],\tau, (\infty)} \to \Fl_K^{[a,b], E(u)}$. We are reduced then to showing that for any $x: R \ra F'$, $\fM_x$ has $p$-adic Hodge type $\leq \mu$ if and only if for any choice of eigenbasis $(\beta^{(j)})$ the corresponding $F'$-point $\Psi(x)$ of $\Fl_K^{[a,b], E(u)}$ lies in $\prod_j (1_{\GL_n/P_j} \times S(\mu_j))$.  We can enlarge the field if necessary so that $F'$ contains a splitting field for $K/\Qp$.  This ensures that the generic fiber of $\Fl_K^{[a,b], E(u)}$ becomes a product over the embeddings $\psi:K \iarrow F'$.   

We first show that the projection to $\GL_n/P_j$ is the identity point.   Consider the Frobenius map
$$
\phi^{(j-1)}_{x, \chi_{s_{j}(n)}}:^{\phz} \fM^{(j-1)}_{x, \chi_{s_{j}(n)}}[1/E_j(u)] \ra \fM^{(j)}_{x, \chi_{s_{j}(n)}}[1/E_j(u)]
$$
which is a map of modules over $(\La'[\![u]\!])[1/p, 1/E_j(u)]$.  Since $p$ is inverted, reduction mod $u$ induces an isomorphism 
$$
^{\phz} \fM^{(j-1)}_{x, \chi_{s_{j}(n)}} \mod u \xrightarrow{\sim} \fM^{(j)}_{x, \chi_{s_{j}(n)}} \mod u.
$$
For a choice of eigenbasis $\beta^{(j)} = \left( f^{(j)}_i \right)$, we would like to show that the image of the filtration on $^{\phz} \fM^{(j-1)}_{x, \chi_{s_{j}(n)}} \mod u$ is the canonical filtration on $\fM^{(j)}_{x, \chi_{s_{j}(n)}} \mod u$ induced by the trivialization (i.e., induced by the eigenbasis).  Concretely, this comes down to the fact that 
$$
\phi^{(j-1)}_{x, \chi_{s_{j}(n)}}(u^{a^{(j)}_{s_j(n)} - a^{(j)}_{s_j(i)}} \otimes f_{s_j(i)}^{(j)} ) \in \Fil^i(\fM^{(j)}_{x, \chi_{s_{j}(n)}} \mod u)
$$
which is equivalent to the commutativity of the (\ref{bigdiagram}).  This shows that $\Psi(x) \in \prod_j 1_{\GL_n/P_j} \times \Gr_{\Res_{(K \otimes_{K_0} F)/F} \GL_n}(F')$. 

By twisting by some power of $E(u)$, we can now reduce to the case of $[a,b] = [0,h]$. Fix an embedding $\sigma_j:K_0 \ra F'$.  We have that $S(\mu_j) = \prod_{\psi:K \ra F'} S(\mu_{j, \psi})$, where the product is over all embeddings $\psi:K\iarrow F'$ which extend $\sigma_j$.  Fix such an embedding $\psi$ and let $\pi_{\psi} := \psi(\pi_K)$. We write $F'[\![u - \pi_{\psi}]\!]$ for the completion of $(\La'[\![u]\!])[1/p]$ at $u - \pi_{\psi}$. 

Let $\Psi(x)_{\psi}$ denote the projection onto the $\Gr_{\GL_n}$ factor corresponding to the embedding $\psi$.  We let $\mathcal{L}_0$ be the standard $F'[\![u - \pi_{\psi}]\!]$-lattice in $\left(F'(\!(u - \pi_{\psi})\!)\right)^n$ corresponding to the image of $\fM_{x,\chi_{s_j(n)}}^{(j)}$ under the trivialization $\gamma^{(j)}$. Then $\Psi(x)_{\psi} \in S(\mu_{j, \psi})$ if and only if the lattice $\mathcal{L}$ given by the $(u - \pi_{\psi})$-adic completion of  
$$
\gamma^{(j)} \circ \phi^{(j-1)}_{x, s_j(n)}\left(^{\phz} \fM^{(j-1)}_{x, \chi_{s_j(n)}} \right) \subset ((\La'[\![u]\!])[1/p])^n 
$$
has relative position less than or equal to $\mu$ relative to $\mathcal{L}_0$. (We note that $\mathcal{L}$ is contained in the standard lattice $\mathcal{L}_0$ by our assumption that $[a,b] = [0,h]$.)  

The question is essentially reduced to one about open Schubert cells in the affine Grassmannian of $F'[\![u - \pi_{\psi}]\!]$-lattices in $\left(F'(\!(u - \pi_{\psi})\!)\right)^n$. We identify an element in $\Gr_{GL_n}$ with such a lattice via the standard lattice $\mathcal{L}_0$. For a cocharacter $\lambda$ of $GL_n$, the Bialynicki-Birula decomposition~\cite{BBdecomp} gives a retraction from the open Schubert cell $S^\circ(\lambda)\subset \Gr_{GL_n}$ to the flag variety $GL_n/P_\lambda$, where 
$$
P_\lambda:= \left\{g\in GL_n |\lim_{(u-\pi_\psi)\to 0}(u-\pi_\psi)^{-\lambda}g(u-\pi_\psi)^\lambda\ \mathrm{exists}\ \right \}.
$$ More precisely, the retraction $S^\circ(\lambda)\to GL_n/P_\lambda$ is induced by the evaluation at $(u-\pi_\psi)=0$ on any element of the $GL_n[\![u-\pi_\psi]\!]$-orbit of $(u-\pi_\psi)^\lambda$. The fact that the evaluation at $0$ map sends $S^\circ(\lambda)$ to $GL_n/P_\lambda$ can be deduced from Lemma 2.3 of~\cite{ngopolo}, where the isotropy group for the orbit of $(u-\pi_\psi)^\lambda$ in the affine Grassmannian is expressed as a semidirect product of $P_\lambda$ and a pro-unipotent algebraic group. 
Moreover, sending a lattice $\mathcal{L}$ to the fitration on $\mathcal{L}/(u-\pi_\psi)\mathcal{L}$ defined by 
$$
\mathrm{Fil}^i(\mathcal{L}/(u-\pi_\psi)\mathcal{L}):=(\mathcal{L}\cap (u-\pi_\psi)^{i}\cdot \mathcal{L}_0)/(u-\pi_\psi)\mathcal{L}
$$
corresponds to the evaluation at $0$ map which sends the open Schubert cell $S^\circ(\lambda)$ to the flag variety $GL_n/P_\lambda$ parametrizing filtrations of type $\lambda$. This can be checked by a direct computation. In particular, if $\mathcal{L}$ corresponds to a point on $S^\circ(\lambda)$, then the filtration $\mathrm{Fil}^{\cdot}(\mathcal{L}/(u-\pi_\psi)\mathcal{L})$ is of type $\lambda$. 

Let $\mathcal{L}$ be the lattice corresponding to $\Psi(x)_{\psi}$. Since
$$
S(\mu_{j,\psi})_{F'}=\sqcup_{\lambda\leq \mu_{j,\psi}}S^\circ(\lambda)_{F'},
$$ 
as we have assumed that $F'$ is large enough, the statement that $\Psi(x)_{\psi}\in S(\mu_{j,\psi})$ is equivalent to the filtration $\mathrm{Fil}^{\cdot}(\mathcal{L}/(u-\pi_\psi)\mathcal{L})$ being of type $\leq \mu_{j, \psi}$. Indeed, $\Psi(x)_{\psi}\in S(\mu_{j,\psi})$ if and only if $\Psi(x)_{\psi}\in S^\circ(\lambda)$ for some cocharacter $\lambda\leq \mu$.  Then $\Psi(x)_{\psi}\in S^\circ(\lambda)$ for some cocharacter $\lambda$ if and only if $\mathrm{Fil}^{\cdot}(\mathcal{L}/(u-\pi_\psi)\mathcal{L})$ is of type $\lambda$. (The if part follows after knowing the only if part for every cocharacter $\lambda$, since the Cartan decomposition tells us that $\Psi(x)_{\psi}$ belongs to some open Schubert cell. The only if part is explained above.)

We are now ready to conclude. Recall that we've set $\cD_{x, \chi_{s_j(n)}} = ^{\phz} \fM^{(j-1)}_{x, s_j(n)}/ E_j(u) ^{\phz} \fM^{(j-1)}_{x, s_j(n)}$. Then $\cD_{x, \chi_{s_j(n)}}$ is isomorphic to the product over those $\psi$ extending $j$ of $\left(\mathcal{L}/(u-\pi_\psi)\mathcal{L}\right)$. This is an isomorphism of filtered $K\otimes_{K_0,\sigma_j}F'$-modules. We see then that $\Psi(x) \in \prod_j 1_{\GL_n/P_j} \times S(\mu_j)$ if and only if $\gr^{\bullet}(\cD_{x, \chi_{s_j(n)}}) \cong V_{\mu'_j} \otimes_F F'$ for some $\mu'_j\leq \mu_j$. By Corollary~\ref{LoverK}, this is equivalent to $\fM_x$ having $p$-adic Hodge type $\leq \mu$.  
\end{proof}

\begin{rmk} When working with the moduli of Kisin modules, one is forced to impose the condition $\leq \mu$ rather than asking for a constant $p$-adic Hodge type $\mu$. On the other hand, Kisin shows in~\cite[Corollary 2.6.2]{KisinPSS} that, for the family of Kisin modules living over the generic fiber of a semistable Galois deformation ring, the $p$-adic Hodge type is locally constant.  However, Kisin's proof that the $p$-adic Hodge type is locally constant uses the comparison with $D_{\mathrm{dR}}$ (see also~\cite[(A.4)]{KisinFM}, where the proof of ~\cite[Corollary 2.6.2]{KisinPSS} is corrected). In particular, the proof relies on the fact that the family of Kisin modules comes from a family of $G_K$-representations, rather than $G_{K_\infty}$-representations.  For general families of finite height Kisin modules over a complete local ring $R$ as above, the $p$-adic Hodge type need not be locally constant on $\Spec R[1/p]$.
\end{rmk}





\subsection{Connections to Galois representations}\label{connection to Galois}

In this subsection, we record two connections to Galois representations in the spirit of  \cite{KisinPSS} and \cite{MFFGS}. This essentially comes down to adding descent datum to the constructions of Kisin in loc. cit. 

Let $R$ be a complete local $\Zp$-algebra. Fix a compatible system of $p$-power roots $\{ \pi_L^{\frac{1}{p}}, \pi_L^{\frac{1}{p^2}}, \ldots  \}$ and let $L_{\infty}$ denote the completion of $L(\pi_L^{\frac{1}{p}}, \pi_L^{\frac{1}{p^2}}, \ldots)$.  We define $K_{\infty}$ to be the completion of the field obtained by adjoining the compatible system of $p$-power roots of $\pi_K$ given by $\{ \pi_L^{\frac{e}{p}}, \pi_L^{\frac{e}{p^2}}, \ldots  \}$. Note that $L_{\infty}$ is Galois over $K_{\infty}$ with $\Gal(L_{\infty}/K_{\infty}) \cong \Gal(L/K) = \Delta$. 

\begin{defn} Let $\cO_{\cE, L}$ be the $p$-adic completion of $(W[\![v]\!])[1/v]$ equipped with Frobenius and an action of $\Delta$ in the natural way. An \emph{\'etale $\phz$-module} over $R$ with descent datum is a finite free $R \widehat{\otimes}_{\Zp} \cO_{\cE, L}$ module $\cM$ equipped with an Frobenius isomorphism $\phi_{\cM}:\phz^*(\cM) \cong \cM$ and a semilinear action $\{ \widehat{g} \}$ of $\Delta$ such that $\phi_{\cM}$ and $\widehat{g}$ commute for all $g \in \Delta$. 
\end{defn}

\begin{prop} There is a functor $\underline{M}_{dd}$ from the category of continuous representations of $G_{K_{\infty}} := \Gal(\overline{K}/K_{\infty})$ on finite free $R$-modules to the category of \'etale $\phz$-modules over $R$ with descent datum. This functor is an equivalence of categories with quasi-inverse $T_{dd}$.  
\end{prop}
\begin{proof} The main content is the equivalence given by the theory of norm fields over $L_{\infty}$ due to Fontaine-Wintenberger (with coefficients \cite[Lemma 1.2.7]{MFFGS}). The addition of descent datum is straightforward (see \cite[\S 2.1.3]{CDM1} for details). 
\end{proof} 

Let $F'$ be a finite extension of $F$ with ring of integers $\Lambda'$. Let $V_{F'}$ be a potentially semistable representation of $G_K$ with Galois type $\tau$ and $p$-adic Hodge type $\mu$.

\begin{prop} \label{latticewithdescent} Let $T_{\La'}$ denote a $G_K$-stable lattice in $V_{F'}$.  Then there exists $\fM_{\La'} \in Y^{\mu, \tau}(\La')$ such that 
$$
\fM_{\La'} \otimes_{W[\![v]\!]} \cO_{\cE, L} \cong \underline{M}_{dd}(T_{\La'}|_{G_{K_{\infty}}})
$$
\end{prop}
\begin{proof} Without descent datum, this is due to Kisin (see Corollary 1.3.15 and Proposition 2.1.5 in \cite{Fcrystals}). We briefly explain how to extend the result to include decent datum. Let $\cM_{\La'} = \underline{M}_{dd}(T_{\La'}|_{G_{K_{\infty}}})$.   

Applying Kisin's results to $T_{\La'}\mid_{G_L}$, we get a finite height lattice $\fM_{\La'} \subset \cM_{\La'}$. The fact that $\fM_{\La'}$ inherits a semilinear action of $\Delta$ from $\cM_{\La'}$ follows from the uniqueness of $\fM_{\La'}$ (\cite[Lemma 2.1.6]{Fcrystals}). The fact that $\fM_{\La'}$ has \emph{type} $\tau$ follows from the $\Gal(L/K)$-equivariance of the identification
$$
(\fM_{\La'}/v \fM_{\La'})[1/p] \cong D_{st}(T_{\La'}[1/p]|_{G_L}) 
$$
from \cite[\S 2.5(1)]{KisinPSS}.

Finally, $\fM_{\La'}[1/p]$ has $p$-adic Hodge type $\mu$ via the identification
$$
\phz^*(\fM_{\La'}[1/p])/P(v) \phz^*(\fM_{\La'}[1/p]) \cong D_{\dR}^*(V_{F'})
$$
from the proof of \cite[Corollary 2.6.2]{KisinPSS}.    
\end{proof}

Let $\rho:\Gal(\overline{K}/K) \ra \GL_n(\La')$ be a lattice in a potentially crystalline representation of $\Gal(\overline{K}/K)$ with Galois type $\tau$ and $p$-adic Hodge type $\mu$.  Let $\rhobar:\Gal(\overline{K}/K) \ra \GL_n(\F')$ denote the reduction of $\rho$ modulo the maximal ideal of $\La'$.    

\begin{cor}  Let $\rhobar$ be as above.   Then there exists $\overline{\fM} \in Y^{\mu, \tau}(\F')$ such that 
$$
T_{dd}(\overline{\fM}) \cong \rhobar|_{\Gal(\overline{K}/K_{\infty})}.
$$
\end{cor} 

We end by considering resolutions of potentially crystalline deformations rings as in \cite{KisinPSS}.  Let $\rhobar:\Gal(\overline{K}/K) \ra \GL_n(\F)$ be a continuous representation.  Let $\mu \in (\Z^n)^{\Hom(K, \overline{\Q}_p)}$ be a cocharacter.  Let $R_{\rhobar}^{\mu, \tau, \cris}$ be the framed potentially crystalline deformation ring with $p$-adic Hodge type $\mu$ and Galois type $\tau$, as constructed by Kisin. 

Let $\mathfrak{m}_R$ denote the maximal ideal of $R_{\rhobar}^{\mu, \tau, \cris}$ and  let $\rho^{\mathrm{univ}}_d:\Gal(\overline{K}/K)  \ra \GL_n(R_{\rhobar}^{\mu, \tau, \cris}/\mathfrak{m}_R^d)$ be the reduction of the universal deformation. Set $\cM_{d} := \underline{M}_{dd}(\rho^{\mathrm{univ}}_d)$.  Define $Y^{\mu, \tau}_{\rhobar, d}$ to be the functor on $R_{\rhobar}^{\mu, \tau, \cris}/\mathfrak{m}_R^d$-algebras $B$ given by 
$$
Y^{\mu, \tau}_{\rhobar, d}(B) := \{ (\fM_B, \alpha) \mid \fM_B \in Y^{\mu, \tau}(B),  \alpha:\fM_B[1/u] \cong \cM_{d} \otimes_{\cO_{\cE, R_{\rhobar}^{\mu, \tau, \cris}/\mathfrak{m}_R^d}} (W \otimes_{\Zp} B)(\!(v)\!) \}
$$

\noindent The functors $Y^{\mu, \tau}_{\rhobar, d}$ are relatively represented by projective schemes over $R_{\rhobar}^{\mu, \tau, \cris}/\mathfrak{m}_R^d$ as subschemes of the affine Grassmannian for $\cM_{d}$ using the same argument as in~\cite{KisinPSS}.  By formal GAGA, there is a projective morphism 
$$
\Theta:Y^{\mu, \tau}_{\rhobar} \ra \Spec R_{\rhobar}^{\mu, \tau, \cris}
$$
reducing to $Y^{\mu, \tau}_{\rhobar, d}$ modulo $\mathfrak{m}_R^d$. 

\begin{thm} The projective morphism
$$
\Theta:Y^{\mu, \tau}_{\rhobar} \ra \Spec R_{\rhobar}^{\mu, \tau, \cris}
$$
is an isomorphism on generic fibers. 
\end{thm}
\begin{proof} 
The proof that $\Theta[1/p]$ is a closed immersion is the same argument as in \cite[Proposition 1.6.4]{KisinPSS} using uniqueness of finite height lattices when $p$ is inverted.  The fact that $\Theta[1/p]$ is an isomorphism is then a consequence of Proposition~\ref{latticewithdescent}.     
\end{proof}

\begin{cor}\label{forgetful map formally smooth}  If $\mu \in (\{0,1\}^n)^{\Hom(K, \overline{\Q}_p)}$, i.e., $R_{\rhobar}^{\mu, \tau, \cris}$ is a potentially Barsotti-Tate deformation ring, then the forgetful map $Y^{\mu, \tau}_{\rhobar} \ra Y^{\mu, \tau}$ is formally smooth.   
\end{cor}
\begin{proof} For $R$ a complete local Noetherian $\La$-algebra, the functor $T_{dd}$ on  $Y^{[0,1], \tau}(R)$ canonically extends to a functor $\widetilde{T}_{dd}$ valued in representations of $G_{K}$ (not just $G_{K_{\infty}}$) such that when $R$ is finite flat over $\La$ the representation is potentially crystalline. To construct $\widetilde{T}_{dd}$, one first associates to $\fM_R \in Y^{[0,1], \tau}(R)$ a strongly divisible module with tame descent as defined in \cite[Definition 7.3.1]{EGS}. The key point is that the monodromy operator is unique and so it commutes with the descent datum.  There is a functor $T_{st, L}$ from strongly divisible modules with tame descent to representations of $G_{K}$ \cite[\S 4]{Sav}.

Formal smoothness is a local property so consider $\overline{\fM} \in Y^{\mu, \tau}_{\rhobar}(\F')$ for $\F'/\F$ finite and the corresponding deformation groupoids $D^{\mu, \tau}_{ \rhobar, \overline{\fM}} \ra D^{\mu, \tau}_{\overline{\fM}}$ describing the local structure of $Y^{\mu, \tau}_{\rhobar}$ and $Y^{\mu, \tau}$ respectively.  As in diagram (2.4.7) of \cite{KisinPSS} , define $D^{\mu, \tau, \square}_{\overline{\fM}}$ to sit in the 2-Cartesian square
\[
\xymatrix{
D^{\mu, \tau, \square}_{\overline{\fM}} \ar[d]^{\widetilde{T}^{\square}_{dd}} \ar[r]^{f.s.} & D^{\mu, \tau}_{\overline{\fM}} \ar[d]^{\widetilde{T}_{dd}} \\
 \mathrm{Spf} \, R_{\rhobar}^{\square} \ar[r]^{f.s.} & D_{\rhobar}  \\
}
\]
where $R_{\rhobar}^{\square}$ is the unrestricted framed deformation ring and $D_{\rhobar}$ is the unrestricted deformation groupoid. (One could also use potentially Barsotti-Tate deformations in place of the unrestricted versions.) 

Finally, the same argument as in Proposition 2.4.8 of \emph{loc. cit.} shows that $D^{\mu, \tau, \square}_{\overline{\fM}} \cong D^{\mu, \tau}_{ \rhobar, \overline{\fM}}$. The key points are that $\widetilde{T}^{\square}_{dd}$ factors through the $\mathrm{Spf} \, R_{\rhobar}^{\mu, \tau, \cris} \subset \mathrm{Spf} \, R_{\rhobar}^{\square}$ and that $Y^{\mu, \tau}$ is flat over $\La$.        
\end{proof} 

\begin{cor} If $\mu \in (\{0,1\}^n)^{\Hom(K, \overline{\Q}_p)}$, then $Y^{\mu, \tau}_{\rhobar}$ is normal and $Y^{\mu, \tau}_{\rhobar} \otimes \F$ is reduced.  
\end{cor} 
\begin{proof} This follows directly from Theorem \ref{locmodels}, Theorem \ref{main thm in body} and Corollary~\ref{forgetful map formally smooth}.  
\end{proof}

\end{document}